\documentclass[11pt, twoside, leqno]{article}
\usepackage{amssymb}
\usepackage{amsmath}
\usepackage{amsthm}
\usepackage{color}
\usepackage{mathrsfs}

\usepackage{indentfirst}

\usepackage{txfonts}

\allowdisplaybreaks

\def\RR {\mathbb{R}}
\def\R2{{\mathbb R}^2}

\def\ls{\lesssim}
\def\gs{\gtrsim}

\def\Xint#1{\mathchoice
{\XXint\displaystyle\textstyle{#1}}%
{\XXint\textstyle\scriptstyle{#1}}%
{\XXint\scriptstyle\scriptscriptstyle{#1}}%
{\XXint\scriptscriptstyle\scriptscriptstyle{#1}}%
\!\int}
\def\XXint#1#2#3{{\setbox0=\hbox{$#1{#2#3}{\int}$ }
\vcenter{\hbox{$#2#3$ }}\kern-.6\wd0}}

\def\dashint{\Xint-}

\def\({\left(}
\def \){ \right)}
\newtheorem{theorem}{Theorem}[section]
\newtheorem{lemma}[theorem]{Lemma}
\newtheorem{corollary}[theorem]{Corollary}
\newtheorem{proposition}[theorem]{Proposition}
\newtheorem{example}[theorem]{Example}
\theoremstyle{definition}
\newtheorem{remark}[theorem]{Remark}

\renewcommand{\appendix}{\par
   \setcounter{section}{0}%
   \setcounter{subsection}{0}%
   \setcounter{subsubsection}{0}%
   \gdef\thesection{\@Alph\c@section}%
   \gdef\thesubsection{\@Alph\c@section.\@arabic\c@subsection}%
   \gdef\theHsection{\@Alph\c@section.}%
   \gdef\theHsubsection{\@Alph\c@section.\@arabic\c@subsection}%
   \csname appendixmore\endcsname
 }

\numberwithin{equation}{section}

\begin{document}

\arraycolsep=1pt

\title{\bf\Large $L^p(\mathbb{R}^2)$-boundedness of Hilbert Transforms and Maximal Functions along Plane Curves with Two-variable Coefficients
\footnotetext{\hspace{-0.35cm} 2010 {\it
Mathematics Subject Classification}. Primary 42B20;
Secondary 42B25.
\endgraf {\it Key words and phrases.} Hilbert transform, maximal function, Littlewood-Paley operator, local smoothing estimate, variable curve.
\endgraf Liang Song is supported by NSF for distinguished Young Scholar of Guangdong Province (No.~2016A030306040);  Haixia Yu is supported by the Fundamental Research Funds for the Central University (No.~20lgpy144).}}
\author{Naijia Liu, Liang Song and Haixia Yu\footnote{Corresponding author.}}

\date{}

\maketitle

\vspace{-0.7cm}

\begin{abstract}
 In this paper, for general plane curves $\gamma$ satisfying some suitable smoothness and curvature conditions, we obtain the single annulus $L^p(\mathbb{R}^2)$-boundedness of the Hilbert transforms $H^\infty_{U,\gamma}$ along the variable plane curves $(t,U(x_1, x_2)\gamma(t))$
and the $L^p(\mathbb{R}^2)$-boundedness of the corresponding maximal functions $M^\infty_{U,\gamma}$, where $p>2$ and $U$ is a measurable function.   The range on $p$  is sharp.   Furthermore, for $1<p\leq 2$, under the additional conditions that $U$ is Lipschitz and making a $\varepsilon_0$-truncation  with $\gamma(2 \varepsilon_0)\leq 1/4\|U\|_{\textrm{Lip}}$,  we also obtain   similar boundedness for these two operators $H^{\varepsilon_0}_{U,\gamma}$ and $M^{\varepsilon_0}_{U,\gamma}$.
\end{abstract}

\section{Introduction}

The main purpose of this article is to study the $L^p(\R2)$-boundedness of the Hilbert transform and  corresponding maximal function along variable plane curve. Let us first recall some backgrounds of this topic. The so-called Zygmund conjecture is a long-standing open problem, which can be stated as follows. Denote

$$M^{\varepsilon_0}_{U}f(x_1,x_2):=\sup_{0<\varepsilon<\varepsilon_0}\frac{1}{2\varepsilon}
\int_{-\varepsilon}^{\varepsilon}|f(x_1-t,x_2-U(x_1,x_2)t)|\,\textrm{d}t.$$

\textbf{Zygmund conjecture:} Let $U:\ \mathbb{R}^2\rightarrow  \mathbb{R}$ be a Lipschitz function and $\varepsilon_0>0$ small enough depending on $\|U\|_{\textrm{Lip}}$. Is the operator $M^{\varepsilon_0}_{U}$ bounded on $L^p(\mathbb{R}^2)$ for some $p\in(1,\infty)$?

Much later Stein \cite{St87} raised the singular integral variant of Zygmund conjecture. Denote

$$H^{\varepsilon_0}_{U}f(x_1,x_2):={\rm p.\,v.}\int_{-\varepsilon_0}^{\varepsilon_0}
f(x_1-t,x_2-U(x_1,x_2)t)\,\frac{\textrm{d}t}{t}.
$$

\textbf{Stein conjecture:} Let $U:\ \mathbb{R}^2\rightarrow  \mathbb{R}$ be a Lipschitz function and $\varepsilon_0>0$ small enough depending on $\|U\|_{\textrm{Lip}}$. Is the operator $H^{\varepsilon_0}_{U}$ bounded on $L^p(\mathbb{R}^2)$ for some $p\in(1,\infty)$?

A trivial fact is that $M^{\varepsilon_0}_{U}$ is bounded on $L^\infty(\mathbb{R}^2)$ and $H^{\varepsilon_0}_{U}$ is unbounded on $L^\infty(\mathbb{R}^2)$ if $U$ is measurable. Furthermore, a counterexample based on a construction of the Besicovitch-Kakeya set shows that we can not expect any $L^p(\mathbb{R}^2)$-boundednesss of $M^{\varepsilon_0}_{U}$ and $H^{\varepsilon_0}_{U}$ for $1<p<\infty$ if $U$ is only  assumed to be  H\"older continuous  $C^\alpha$  with $\alpha<1$.  Both Zygmund conjecture and Stein conjecture are very difficult conjectures. Indeed, it is known that if the Stein Conjecture holds for $C^2$ vector fields, then Carleson's Theorem on the pointwise convergence of Fourier series \cite{Ca} would follow.

Let us state some partial progresses toward understanding the above two open problems. For any real analytic function $U$, Bourgain \cite{Bour}  first  obtained the $L^2(\mathbb{R}^2)$-boundedness of $M^{\varepsilon_0}_{U}$, and the $L^p(\mathbb{R}^2)$-boundedness can also been obtained with some standard modifications. The corresponding result for $H^{\varepsilon_0}_{U}$ can be found in \cite{SS}.  For $U\in C^\infty$ with some additional curvature conditions, Christ, Nagel, Stein and Wainger \cite{CNSW} proved the $L^p(\mathbb{R}^2)$-boundedness of $H^{\varepsilon_0}_{U}$ and $M^{\varepsilon_0}_{U}$ for $p>1$.
Furthermore, Lacey and Li \cite{LL2} established the $L^2(\mathbb{R}^2)$-boundedness of $H^{\varepsilon_0}_{U}$ if $U\in C^{\alpha}$ with $\alpha>1$ and a suitable Kakeya maximal operator is bounded on $L^2(\mathbb{R}^2)$. For further progress towards these two conjectures, we refer to \cite{B1,B2,G2,G1,G3} and the references therein. Moreover, Lacey and Li \cite{LL1} brought tools from time-frequency analysis into the problem of Hilbert transforms along vector fields, which made the major breakthrough in terms of the regularity of $U$. To state their results, we first introduce some definitions.

Let $\psi:\ \mathbb{R}\rightarrow\mathbb{R}$ be a smooth function supported on $\{t\in \mathbb{R}:\ 1/2\leq |t|\leq 2\}$ with the property that $0\leq \psi(t)\leq 1$ and $\Sigma_{k\in \mathbb{Z}} \psi_k(t)=1$ for any $t\neq 0$, where $\psi_k(t):=\psi (2^{-k}t)$. Here and hereafter, for any $k\in \mathbb{Z}$ and $j=1,2$,  we denote $P^{(j)}_k$  the \emph{Littlewood-Paley projection in the $j$-th variable corresponding to $\psi_k$}
\begin{align}\label{1.0}
P^{(1)}_kf(x_1,x_2):=\int_{-\infty}^{\infty}  f(x_1-z,x_2)\check{\psi}_k(z)\,\textrm{d}z, \qquad
P^{(2)}_kf(x_1,x_2):=\int_{-\infty}^{\infty}  f(x_1,x_2-z)\check{\psi}_k(z)\,\textrm{d}z.
\end{align}

Lacey and Li obtained the following results.
\begin{theorem}\label{main result 6} (\cite{LL1}) Let $U:\ \mathbb{R}^2\rightarrow  \mathbb{R}$ be a measurable function. For any $p\geq 2$, there exists a positive constant $\tilde{C}_p$, independent of $k$, such that
\begin{align*}
\left\|H^\infty_{U}P^{(2)}_kf\right\|_{L^{2,\infty}(\mathbb{R}^{2})}\leq \tilde{C}_2 \left\|P^{(2)}_kf\right\|_{L^{2}(\mathbb{R}^{2})} \quad \quad {\rm for \ any } \ k\in {\mathbb Z},
\end{align*}
and for any $p>2$,
\begin{align*}
\left\|H^\infty_{U}P^{(2)}_kf\right\|_{L^{p}(\mathbb{R}^{2})}\leq \tilde{C}_p \left\|P^{(2)}_kf\right\|_{L^{p}(\mathbb{R}^{2})} \quad \quad {\rm for \ any } \ k\in {\mathbb Z}.
\end{align*}
\end{theorem}

\noindent It should be pointed out that the weak $L^2(\mathbb{R}^2)$ estimates are sharp for measurable vector fields $U$.

Let $\gamma:\ \mathbb{R}\rightarrow  \mathbb{R}$ be a continuous curve with $\gamma(0)=0$. We  next will focus  on the study of the Hilbert transform and  corresponding maximal function along the variable plane curve $(t,U(x_1, x_2)\gamma(t))$.  Similarly, we define:

$$H^{\varepsilon_0}_{U,\gamma}f(x_1,x_2):=\mathrm{p.\,v.}\int_{-\varepsilon_0}^{\varepsilon_0}
f\left(x_1-t,x_2-U(x_1,x_2)\gamma(t)\right)\,\frac{\textrm{d}t}{t};$$

$$M^{\varepsilon_0}_{U,\gamma}f(x_1,x_2):=\sup_{0<\varepsilon<\varepsilon_0}
\frac{1}{2\varepsilon}\int_{-\varepsilon}^{\varepsilon}\left|f(x_1-t,x_2-U(x_1,x_2)\gamma(t))\right|
\,\textrm{d}t.$$

Denote ${[t]}^\alpha:={|t|}^\alpha  \ {\rm or}\  {\rm sgn(t)}{|t|}^\alpha.$  The following results were obtained in \cite{MR,GHLR}.

 \begin{theorem}\label{thm of M-R and GHLR}
 Let $\alpha>0$ and $\alpha\neq 1$. Suppose that $U: \R2\to \RR$ is measurable, then for any $2<p<\infty$,
 \begin{align}\label{M on t alpha}
 \left\|M^{\infty}_{U,[t]^{\alpha}}f\right\|_{L^p(\R2)}\leq C_{p,\alpha}\|f\|_{L^p(\R2)}
\end{align}
 \noindent and
 \begin{align}\label{H on t alpha}
 \left\|H^{\infty}_{U,[t]^{\alpha}}P_k^{(2)}f\right\|_{L^p(\R2)}\leq \tilde{C}_{p,\alpha}\left\|P_k^{(2)}f\right\|_{L^p(\R2)},
\end{align}
 where $C_{p,\alpha},\tilde{C}_{p,\alpha}$ are positive constants that  depend only on $p$ and $\alpha$.
 \end{theorem}
We note that (\ref{M on t alpha}) was first proved by Marletta and Ricci \cite{MR}, in which the authors used Bourgain's result on the circular maximal operator \cite{Bour86} as a black box. Later, Guo, Hickman, Lie and Roos \cite{GHLR} adopted another approach that is more self-contained to reprove (\ref{M on t alpha}). They also proved (\ref{H on t alpha}). Moreover, under the condition that $U$ is Lipschitz,  Guo {\it et al.}  \cite{GHLR} obtained  the following result for $1<p\leq 2$.

 \begin{theorem}\label{M p<2 t alpha}
 Let $\alpha>0$ and $\alpha\neq 1$. Suppose that $U: \R2\to \RR$ is Lipschitz, then  there exists $\varepsilon_0>0$ depending only on $\|U\|_{\textrm{Lip}}$ and $\alpha$ such that
 $$
 \left\|M^{\varepsilon_0}_{U,[t]^{\alpha}}f\right\|_{L^p(\R2)}\leq C_{p,\alpha}\|f\|_{L^p(\R2)}
 $$
for any $1<p\leq 2$, where $C_{p,\alpha}$ is a positive constant that depends only on $p$ and $\alpha$.
 \end{theorem}

\smallskip

As the development of  the Hilbert transforms along curves $(t,\gamma(t))$, {\it an interesting question is   whether these results of Theorems \ref{thm of M-R and GHLR} and \ref{M p<2 t alpha} can be extended to more general curves $\gamma$  obeying some suitable smoothness and curvature conditions.}

In the present paper,  we have obtained some  results for this question. Firstly,  we  state the needed conditions on the curves.

\smallskip

{\bf Hypothesis on curves (H.).} Assume $\gamma\in C(\mathbb{R})\bigcap C^{N}(\mathbb{R}^+)$  with $N\in\mathbb{N}$ large enough and $\gamma(0)=0$. And $\gamma$ is either odd or even, and increasing on $\mathbb{R}^+$. Moreover, $\gamma$ satisfies the following three conditions:

\begin{enumerate}\label{curve gamma}
  \item[\rm(i)] there exists a positive constant $C_1$ such that $|(\frac{\gamma'}{\gamma''})'(t)|\geq C_1$ for any $t\in \mathbb{R}^+$;
  \item[\rm(ii)] there exist positive constants $\{C^{(j)}_{2}\}_{j=1}^{2}$ such that $|\frac{t^{j}\gamma^{(j)}(t)}{\gamma(t)}|\geq C^{(j)}_{2}$  for any $t\in \mathbb{R}^+$;
  \item[\rm(iii)] there exist positive constants $\{C^{(j)}_{3}\}_{j=1}^{N}$ such that $|\frac{t^{j}\gamma^{(j)}(t)}{\gamma(t)}|\leq C^{(j)}_{3}$ for any $t\in \mathbb{R}^+$.
\end{enumerate}

\begin{remark}\label{remark 1}
We give some remarks for the hypothesis above.
  (i) of {\bf(H.)}, including some other similar forms, has been applied in the study of  the Hilbert transforms along the variable curves $(t, u(x_1)\gamma(t))$ with one-variable coefficients (e.g., \cite{Lie, CZx}),  and  the bilinear Hilbert transform along the curves $(t,\gamma(t))$ (\cite{LY3}).  (ii) and (iii) of {\bf(H.)} imply the following ``doubling property" of the curves. For $t>0$, there holds
 \begin{align}\label{eq:1.5}
e^{C^{(1)}_{2}/2}\leq \frac{\gamma(2t)}{\gamma(t)}\leq e^{C^{(1)}_{3}}
  \end{align}
and
\begin{align}\label{eq:1.6}
\textrm{either} \quad  e^{C^{(2)}_{2}/2 C^{(1)}_{3}}\leq \frac{\gamma'(2t)}{\gamma'(t)}\leq e^{C^{(2)}_{3}/C^{(1)}_{2}}\quad  \textrm{or}\quad e^{-C^{(2)}_{3}/C^{(1)}_{2}}\leq \frac{\gamma'(2t)}{\gamma'(t)}\leq e^{-C^{(2)}_{2}/2 C^{(1)}_{3}}.
  \end{align}
In fact, it follows from (ii) and (iii) of {\bf(H.)}  that
  \begin{align*}
C^{(1)}_{2}\leq \frac{t\gamma'(t)}{\gamma(t)}\leq C^{(1)}_{3}\quad \textrm{and} \quad C^{(2)}_{2}/ C^{(1)}_{3}\leq \left|\frac{t\gamma''(t)}{\gamma'(t)}\right|\leq C^{(2)}_{3}/C^{(1)}_{2} \quad  {\rm for \ any \ }t\in \mathbb{R}^+.
  \end{align*}
Let $F(t):=\ln \gamma(t)$ for any $t\in \mathbb{R}^+$, we then have $C^{(1)}_{2}/t\leq F'(t)\leq C^{(1)}_{3}/t$ for any $t\in \mathbb{R}^+$. On the other hand, by the Lagrange mean value theorem, there exists  $\theta\in[1,2]$ such that $F(2t)-F(t)=F'(\theta t)t$. Hence, we have $C^{(1)}_{2}/2\leq F(2t)-F(t)\leq C^{(1)}_{3}$, which further implies \eqref{eq:1.5}. Similarly, if consider $G(t):=\ln \gamma'(t)$ for any $t\in \mathbb{R}^+$, we can get \eqref{eq:1.6}. We note that  some similar ``doubling properties" of curves have  occurred in obtaining the boundedness of the Hilbert transforms along  curves $(t,\gamma(t))$ (cf. \cite{CCCD}).
\end{remark}

\begin{example}
Let us list some examples of curves satisfying {\bf(H.)}. Since $\gamma(t)$ is odd or even and $\gamma(0)=0$, we write only the part for $t>0$.
\begin{enumerate}
\item[\rm(1)] for any $t>0$, $\gamma(t):=t^\alpha$, where $\alpha\in(0,\infty)$ and $\alpha\neq1$;
\item[\rm(2)] for any $k\in \mathbb{N}$ and $t>0$, $\gamma(t):=\sum\limits_{i=1}^{k} t^{\alpha_i}$, where either $\alpha_i\in(0,1)$ for all $i=1,2, \cdots, k$, or $\alpha_i>1$  for all $i=1,2, \cdots, k$.
\item[\rm(3)] for any $t>0$, $\gamma(t):=t^\alpha\log(1+t)$, where $\alpha>1$;
\item[\rm(4)] for any $t>0$, $\gamma(t):=(t\sin t)\mathbf{1}_{\{0<t<\varepsilon_0\}}(t)$, or $(t-\sin t)\mathbf{1}_{\{0<t<\varepsilon_0\}}(t)$, or $(1-\cos t)\mathbf{1}_{\{0<t<\varepsilon_0\}}(t)$ , where $\varepsilon_0$ is small enough.
\end{enumerate}
\end{example}

Next, we will state  two main results of this paper.  The first one is as follows.

\medskip

\noindent {\bf Theorem A.} \  Let $U:\ \mathbb{R}^2\rightarrow  \mathbb{R}$ be a measurable function, and the curve $\gamma$ satisfies {\bf(H.)}.
Then, for any  $p>2$, there exists a positive constant $C$, independent of $U$,  such that the following estimates hold for all $f\in L^{p}(\mathbb{R}^{2})$,
\begin{enumerate}
\item [\rm(i)] \quad $\left\|H^\infty_{U,\gamma}P^{(2)}_kf\right\|_{L^{p}(\mathbb{R}^{2})}\leq C \left\|P^{(2)}_kf\right\|_{L^{p}(\mathbb{R}^{2})}$, where $C$ does not depend on  $k\in \mathbb{Z};$

\item [\rm(ii)] \quad $\left\|M^\infty_{U,\gamma}f\right\|_{L^{p}(\mathbb{R}^{2})}\leq C \left\|f\right\|_{L^{p}(\mathbb{R}^{2})}$.
\end{enumerate}

 Observe that Theorem A  is the generalization of  Theorem \ref{thm of M-R and GHLR}  from the special curve $[t]^\alpha$ to more general curves $\gamma(t)$. As for (i) of Theorem A, the reason that we consider single annulus $L^p(\mathbb{R}^2)$-boundedness, in place of $L^p(\mathbb{R}^2)$-boundedness, is that the latter fails for every $p\in(1,\infty)$ even in the case  $\gamma(t)=t^2$. It will follow from a straightforward modification of Karagulyan's counter-example in the case  $\gamma(t)=t$ (\cite{K}). The range on $p$ in Theorem A is sharp. It can be seen  that  Theorem A  fails for $p\leq 2$,  if we assume $f$ to be the characteristic function of the unit ball.

 As a direct corollary of Theorem A and linearization,  we have the results below about the \emph{directional Hilbert transforms} along  curves $\gamma$.

\begin{corollary}
If  $\gamma$ satisfies {\bf(H.)}. Then, for any  $p>2$, there exists a positive constant $C$, such that the following estimates hold for all $f\in L^{p}(\mathbb{R}^{2})$,
\begin{enumerate}
\item [\rm(i)] \quad $\left\|  \sup_{\lambda\in \mathbb{R}} \left| H^\infty_{\lambda,\gamma}P^{(2)}_kf\right | \right\|_{L^{p}(\mathbb{R}^{2})}\leq C \left\|P^{(2)}_kf\right\|_{L^{p}(\mathbb{R}^{2})}$\quad  uniformly in $k\in \mathbb{Z}$;
\item [\rm(ii)] \quad $\left\|  \sup_{\lambda\in \mathbb{R}} \left| M^\infty_{\lambda,\gamma}f\right | \right\|_{L^{p}(\mathbb{R}^{2})}\leq C \left\|f\right\|_{L^{p}(\mathbb{R}^{2})}$.
\end{enumerate}
\end{corollary}

\medskip

The second main result of this paper, Theorem B, will study the analogue of Theorem A for the case $1<p\leq 2$ under the additional conditions that $U$ is Lipschitz and making a $\varepsilon_0$-truncation. (ii) of Theorem B is also a extension of Theorem \ref{M p<2 t alpha} to more general curves $\gamma$.

\noindent {\bf Theorem B.} Let $U:\ \mathbb{R}^2\rightarrow  \mathbb{R}$ be a Lipschitz function and the curve $\gamma$ satisfies {\bf(H.)}.  Then, for any $1<p\leq 2$, there exist constants $C>0$ and $\varepsilon_0>0$ with $\gamma(2 \varepsilon_0)\leq 1/4\|U\|_{\textrm{Lip}}$,  such that the following estimates hold for all $f\in L^{p}(\mathbb{R}^{2})$,
\begin{enumerate}
\item [\rm(i)] \quad $\left\|H^{\varepsilon_0}_{U,\gamma}P^{(2)}_kf\right\|_{L^{p}(\mathbb{R}^{2})}\leq C \left\|P^{(2)}_kf\right\|_{L^{p}(\mathbb{R}^{2})}$;
\item [\rm(ii)] \quad $\left\|M^{\varepsilon_0}_{U,\gamma}f\right\|_{L^{p}(\mathbb{R}^{2})}\leq C \left\|f\right\|_{L^{p}(\mathbb{R}^{2})}$,
\end{enumerate}
where  $C$ is independent of $U$ and $k\in \mathbb{Z}$.

\begin{remark}\label{remark 2}

We now point out  the differences of  proofs between  our theorems and Theorems \ref{thm of M-R and GHLR}  and  \ref{M p<2 t alpha}.  In the homogeneous curve case   $\gamma(t)=[t]^\alpha$,
 the following  special property
  \begin{align}\label{homogenous}
  \gamma(ab)=\gamma(a)\gamma(b), \quad  {\rm for \ any \ } a>0, b>0,
 \end{align}
 plays a very important role in the proofs of Theorems \ref{thm of M-R and GHLR}  and  \ref{M p<2 t alpha} (\cite{GHLR}). More precisely, it is convenient to use
\begin{align}\label{unit decomposition}
1=\Sigma_{l\in \mathbb{Z}} \psi(2^{-l}(u_z^{(0)})^\beta t),~\textrm{see}~\cite[(4.9)]{GHLR},\quad \textrm{and} \quad 1=\Sigma_{l\in \mathbb{Z}} \psi(2^{-l}2^{\frac{v_z}{\alpha}}(u_z^{(0)})^\beta t),~\textrm{see}~ \cite[(5.2)]{GHLR}
\end{align}
to split these operators considered. Here  the purpose of adding $2^{\frac{v_z}{\alpha}}$ in \cite[(5.2)]{GHLR} is to make the terms of $u_z$ and $\gamma(2^{-\frac{v_z}{\alpha}})$  cancel out, after applying this special property \eqref{homogenous}. However,  in the general curve case, we can't continue to use these partition of unity \eqref{unit decomposition} to split our operators, since $\gamma(2^l(u_z^{(0)})^{-\beta}t)\neq \gamma(2^l) \gamma((u_z^{(0)})^{-\beta})\gamma(t)$ and $\gamma(2^l2^{-\frac{v_z}{\alpha}}(u_z^{(0)})^{-\beta} t)\neq \gamma(2^l)\gamma(2^{-\frac{v_z}{\alpha}}) \gamma((u_z^{(0)})^{-\beta})\gamma(t)$ in general.  As a result, we have to use the classical partition of unity, i.e.,
\begin{align}\label{classical unit}
1=\Sigma_{l\in \mathbb{Z}} \psi(2^{-l}t),
\end{align}
to split our operators, though it will greatly increase the difficulty  of  the proof.  Indeed, even though we have used the classical partition of unity \eqref{classical unit},  we still will encounter the difficulty of $\gamma(2^lt)\neq \gamma(2^l)\gamma(t)$.  In order to overcome this difficulty and   separate $2^l$ from $\gamma(2^lt)$,  we replace  $\gamma(2^lt)$ by $\Gamma_l(t):=\gamma(2^lt)/\gamma(2^l)$ in  our proof, which is based on the observation that the main properties of   $\Gamma_l$ are very similar to  that of $\gamma$.

On the other hand,   based on this special property  \eqref{homogenous} of $\gamma(t)=[t]^\alpha$,  Guo {\it et al.} \cite{GHLR} can reduce the proof of \eqref{H on t alpha} to obtaining   a local smoothing estimate to
$$A_{u,t^\alpha} f(x,y):=\int_{-\infty}^{\infty} f(x-ut,y-ut^\alpha)\psi_0(t)\, \textrm{d}t.
$$
In general curves case  considered, we may not reduce our theorems  to some local smoothing estimate to  $A_{u,\gamma(t)}$  since the lack of this special property  \eqref{homogenous}.  Furthermore,  the critical point of the phase function in $A_{u,t^\alpha}$ is independent of $u$, but  the critical point in  general curve case will depend on $u$, which  leads to   essential  difficulties.
In this paper,  we   combine  the theory of oscillatory integrals and interpolation   with the result of   Beltran, Hickman and Sogge \cite[Proposition 3.2]{B}  to show a kind of variable coefficient local smoothing estimate.  Our proofs of Theorems A and B  rely on  this  variable coefficient local smoothing estimate,  the Littlewood-Paley theory and a bootstrapping  argument similar that of Nagel, Stein and Wainger \cite{NSW}.
\end{remark}

We should  point out that the study of the boundedness properties of the Hilbert transforms along curves when $U$ is a constant,  first appeared in the work of Jones \cite{J} and Fabes and Rivi\`ere \cite{FR}  for studying the behavior of the constant coefficient parabolic differential operators. Later, the study has been extended to  more general classes of curves; see, for example, \cite{SW,NVWW,CNVWW,CCVWW,CVWW}.

In the case of $U(x_1,x_2)=u(x_1)$, many important results have been obtained. For example,  when $U(x_1,x_2)=x_1$,    Carbery,  Wainger and  Wright \cite{CWW} obtained the $L^{p}(\mathbb{R}^2)$-boundedness of $H^\infty_{U,\gamma}$ and $M^\infty_{U,\gamma}$ for $p\in(1,\infty)$,  where $\gamma\in C^{3}(\mathbb{R})$ is either an odd or even, convex  on $\mathbb{R}^+$ satisfying $\gamma(0)=\gamma'(0)=0$ and
$\frac{t\gamma''(t)}{\gamma'(t)}$ is decreasing and bounded below on $\mathbb{R}^+$.
It is worth  noting that these conditions allow the curve to be flat at the origin (e.g. $\gamma(t)=e^{-1/t^2}$). Then, Bennett \cite{BJ} extended the $L^2(\mathbb{R}^2)$ results of \cite{CWW}  to the case  $U(x_1,x_2)=P(x_1)$, where $P(x_1)$ is a polynomial. Some other related results about the one-variable coefficient case, we refer to \cite{CP,Lie,CZx,Y}.

For the general two-variable  case,  we would like to  mention more useful results besides the relevant results introduced at the beginning of this paper.  Seeger and Wainger \cite{SeW} obtained the $L^{p}(\mathbb{R}^2)$-boundedness of $H^\infty_{U,\gamma}$ and $M^\infty_{U,\gamma}$ for $p\in(1,\infty)$, where $U(x_1,x_2)\gamma(t)$ was written as $\Gamma(x_1,x_2,t)$, under some convexity and doubling hypothesis about $\Gamma$. Recently, for $\gamma(t):=[t]^\alpha$ (where  $0<\alpha<1$ or $\alpha>1$),  Di Plinio,  Guo,  Thiele and  Zorin-Kranich \cite{DGTZ} obtained the $L^{p}(\mathbb{R}^2)$-boundedness, $p\in(1,\infty)$, of $H^{\varepsilon_0}_{U,\gamma}$ from some positive constant $\varepsilon_0$ and a Lipschitz function $U:\ \mathbb{R}^2\rightarrow  \mathbb{R}$ satisfying $\|U\|_{\textrm{Lip}}\ls 1$, where  Jones's beta numbers from \cite{Jo} play an important role in their proof.

\smallskip

The layout of the article is as follows. In Section 2, we get a kind of variable coefficient local smoothing estimate based on Beltran, Hickman and Sogge \cite[Proposition 3.2]{B}. In Subsection 3.1, we show (i) of Theorem A, whose proof relies  heavily on the variable coefficient local smoothing estimate in Section 2. In Subsection 3.2, we prove (ii) of Theorem A.  Its proof is less difficult  than (i) of Theorem A,  since  the  maximal function in the proof don't need  a summation process relative to  Hilbert transform. In Section 4, we first obtain the $L^{p}(\mathbb{R}^2)$-boundedness, $p\in(1,\infty)$,  of the maximal function associated with plane curve $(t,2^j\gamma(t))$ in lacunary coefficient, which will play a key role in our proof. Then we  prove (i) of Theorem B in Subsection 4.1 by bootstrapping an iterated interpolation argument in the spirit of Nagel,  Stein and  Wainger \cite{NSW}. Finally, in Subsection 4.2  we prove (ii) of Theorem B by following the approach of (ii) of Theorem A and (i) of Theorem B.

\smallskip

Throughout this paper, the letter ``$C$"   will denote (possibly different) constants that are independent of the essential variables.  $a\ls b$ (or $a\gs b$) means that there exist a positive constant $C$ such that $a\le Cb$ (or $a\ge Cb$). $a\approx b$ means $a\ls b$ and $b\ls a$. $\hat{f}$ and $\check{f}$ shall denote the {Fourier transform} and the {inverse Fourier transform}  of $f$, respectively. For $1<q\leq\infty$, we will denote $q'$ the adjoint number of $q$, i.e. ${1}/{q}+ {1}/{q'}=1$.  For any set $E$, we use $\mathbf{1}_E$ to denote the {characteristic function} of $E$. For any $a<b$, we will denote $\dashint_{a}^bf(t)\,\textrm{d}t:=\frac{1}{|b-a|} \int_{a}^bf(t)\,\textrm{d}t$.

\section{A key local smoothing estimate}

In this section, the aim is to prove a kind of variable coefficient local smoothing estimate, which will be used to prove Theorem A in Section 3.  In the proof of  this kind of variable coefficient local smoothing estimate,   \cite[Proposition 3.2]{B} is  a powerful  tool.

Let $\Gamma_{r}(t):=\gamma(2^rt)/ \gamma(2^r)$ with $r\in \mathbb{R}$.  Assume that $\phi,\psi$ are  smooth functions supported on $\{t\in \mathbb{R}:\ 1/2\leq |t|\leq 2\}$.   For any $u\in [1,2)$,  denote
 \begin{align}\label{eq:20.a}
T_{u} f(x_1,x_2):=\int_{-\infty}^{\infty}f\big(x_1-t,x_2- u\Gamma_{r}(t)\big)\phi(t)\,\textrm{d}t
\end{align}
and
\begin{align}\label{eq:20.b}
P_\kappa f(x_1,x_2):=\int_{-\infty}^{\infty}\int_{-\infty}^{\infty}e^{i(x_1\xi+x_2\eta)}\hat{f}(\xi,\eta)\psi\left(2^{-\kappa }(\xi^2+\eta^2)^{\frac{1}{2}}\right)\,\textrm{d}\xi\,\textrm{d}\eta.
\end{align}

\begin{proposition}\label{proposition 2.1}
For any $p>2$, there exist positive constants $\delta$ and $C$, independent of $r$ and $\kappa$, such that
\begin{align}\label{eq:20.13}
\left\|\sup_{u\in [1,2)}\left|T_{u}P_\kappa f\right|\right\|_{L^{p}(\mathbb{R}^{2})}\leq C 2^{-\delta \kappa}\|f\|_{L^{p}(\mathbb{R}^{2})} \quad \textrm{for}~ \textrm{all} ~\kappa\in \mathbb{N}.
\end{align}
\end{proposition}

\smallskip

Firstly, by using the fundamental theorem of calculus to $|T_u|^p$ and $\textrm{H}\ddot{\textrm{o}}\textrm{lder}$'s inequality, we have
\begin{align*}
\sup_{u\in [1,2)}|T_ug|^p\ls |T_1g|^p+\left(\int_1^2|T_ug|^p\,\textrm{d}u \right)^{\frac{1}{p'}}\left(\int_1^2|\partial_uT_ug|^p\,\textrm{d}u \right)^{\frac{1}{p}},
\end{align*}
 which  enables us to reduce the proof of   \eqref{eq:20.13} to  prove the following \eqref{eq:20.15}-\eqref{eq:20.17}:
\begin{align}\label{eq:20.15}
\left\|T_1P_\kappa f\right\|_{L^{p}(\mathbb{R}^{2})}\ls 2^{-\delta \kappa}\|f\|_{L^{p}(\mathbb{R}^{2})}\quad \textrm{for}~ \textrm{all} ~\kappa\in \mathbb{N}~ \textrm{and}~p\in (2,\infty);
\end{align}
\begin{align}\label{eq:20.16}
\left(\int_1^2 \left\|T_uP_\kappa f\right\|^p_{L^{p}(\mathbb{R}^{2})}\,\textrm{d}u\right)^{\frac{1}{p}}\ls 2^{-\left(\delta+\frac{1}{p}\right) \kappa}\|f\|_{L^{p}(\mathbb{R}^{2})}\quad \textrm{for}~ \textrm{all} ~\kappa\in \mathbb{N}~ \textrm{and}~p\in (2,\infty);
\end{align}
and
\begin{align}\label{eq:20.17}
\left(\int_1^2 \left\|\partial_uT_uP_\kappa f\right\|^p_{L^{p}(\mathbb{R}^{2})}\,\textrm{d}u\right)^{\frac{1}{p}}\ls 2^{-\left(\delta+\frac{1}{p}\right) \kappa}2^\kappa\|f\|_{L^{p}(\mathbb{R}^{2})}\quad \textrm{for}~ \textrm{all} ~\kappa\in \mathbb{N}~ \textrm{and}~p\in (2,\infty).
\end{align}

\bigskip
Next, we provide the proofs of \eqref{eq:20.15}-\eqref{eq:20.17} in turn.
\bigskip

\noindent {\bf Proof of \eqref{eq:20.15}:} It is clear that
\begin{align*}
\left\|T_1P_\kappa f\right\|_{L^{\infty}(\mathbb{R}^{2})}\ls \|f\|_{L^{\infty}(\mathbb{R}^{2})}.
\end{align*}
By interpolation, it only needs to prove \eqref{eq:20.15} for $p=2$.

Consider the case of $p=2$.  Let the multiplier of $T_1P_\kappa f$ be
\begin{align*}
m_1(\xi,\eta):=\psi\left(2^{-\kappa }(\xi^2+\eta^2)^{\frac{1}{2}}\right)\int_{-\infty}^{\infty}e^{-it\xi-i\Gamma_{r}(t)\eta}\phi(t)\,\textrm{d}t,
\end{align*}
and the corresponding phase function be $\Phi_1(t):=-t\xi-\Gamma_{r}(t)\eta$, we then have $\Phi'_1(t)=-\xi-\Gamma'_{r}(t)\eta$ and $\Phi''_1(t)=-\Gamma''_{r}(t)\eta$.  Assume  $|\xi|\approx 2^\mathcal{M}$ and $|\eta|\approx 2^\mathcal{N}$, then we have $\max\{\mathcal{M},\mathcal{N}\}\approx\kappa$. From (ii) of {\bf(H.)} and \eqref{eq:1.5}, we have $|\Gamma''_{r}(t)| \gs 1$ with a bound independent of $r$, which further yields $|\Phi''_1(t)|\gs 2^\mathcal{N}$. We apply van der Corput's lemma to obtain
\begin{align}\label{eq:20.18}
|m_1(\xi,\eta)|\ls 2^{-\frac{1}{2}\mathcal{N}}.
\end{align}

If $\max\{\mathcal{M},\mathcal{N}\}=\mathcal{N}$, we use Plancherel's theorem to obtain \eqref{eq:20.15} with $\delta=1/2$ and $p=2$. Consider the case $\max\{\mathcal{M},\mathcal{N}\}=\mathcal{M}$. If $\mathcal{M}\leq 32C_3^{(1)}e^{C^{(1)}_{3}}\mathcal{N}$, from \eqref{eq:20.18}, we also obtain \eqref{eq:20.15} with $p=2$ and some $\delta>0$; If $\mathcal{M}\geq 32C_3^{(1)}e^{C^{(1)}_{3}}\mathcal{N}$, by (iii) of {\bf(H.)} and \eqref{eq:1.5}, we have $|\Gamma'_{r}(t)|\leq 2 C_3^{(1)}e^{C^{(1)}_{3}}$. Therefore, $|\Phi'_1(t)|\geq |\xi|/2-|\Gamma'_{r}(t)\eta|\gs 2^\mathcal{M}\approx 2^{\kappa}$. This, combined with van der Corput's lemma and the fact that $\Phi'_1(t)$ is monotonic, implies
$
|m_1(\xi,\eta)|\ls  2^{-\kappa}.
$
Applying Plancherel's theorem again, we  obtain \eqref{eq:20.15} with $\delta=1$ and $p=2$.

We complete  the proof of \eqref{eq:20.15}.

\bigskip

Before  giving the proof of \eqref{eq:20.16}, we state and show  some necessary lemmas.

\begin{lemma}\label{lemma 2.1}
Let $t\in [1/2,2]$. We have the following inequalities hold uniformly in $r$,
\begin{enumerate}
  \item[\rm(i)] $  e^{-C^{(1)}_{3}}\leq\Gamma_{r}(t)\leq e^{C^{(1)}_{3}}$;
  \item[\rm(ii)] $ C^{(1)}_2/2e^{C^{(1)}_{3}} \leq|\Gamma_{r}'(t)|\leq 2e^{C^{(1)}_{3}}C^{(1)}_3$;
  \item[\rm(iii)] $C^{(2)}_2/4e^{C^{(1)}_{3}}\leq|\Gamma_{r}''(t)|\leq 4e^{C^{(1)}_{3}}C^{(2)}_{3}$;
  \item[\rm(iv)] $|\Gamma_{r}^{(j)}(t)|\leq  2^{j} e^{C^{(1)}_{3}}C^{(j)}_{3}$ \quad \textrm{for} \textrm{all} $2\leq j\leq N$;
  \item[\rm(v)] $|((\Gamma_{r}')^{-1})^{(k)}(t)|\lesssim 1$ \quad \textrm{for} \textrm{all} $0\leq k< N$, \textrm{where} $(\Gamma_{r}')^{-1}$ \textrm{is} \textrm{the} \textrm{inverse} \textrm{function} \textrm{of} $\Gamma_{r}'$.
\end{enumerate}
\end{lemma}

\begin{proof}[Proof of Lemma \ref{lemma 2.1}.] (i), (ii), (iii) and (iv) are straightforward to verify by \eqref{eq:1.5}, (ii) and (iii) of {\bf(H.)}.  As for (v),  noting that $\gamma'$ is strictly monotonic on $\mathbb{R}^+$ with $\gamma'(\mathbb{R}^+)=\mathbb{R}^+$ from (ii) of {\bf(H.)} and \eqref{eq:1.6}, then there exists the inverse function $(\Gamma_{r}')^{-1}$ of $\Gamma_{r}'$. It is easy to see that \eqref{eq:1.6} and Lemma \ref{lemma 2.1}(ii) imply $|(\Gamma_{r}')^{-1}(t)|\approx 1$. For $k=1$, as Lemma \ref{lemma 2.1}(iii), we have
\begin{align}\label{eq:20.20}
|\Gamma_{r}''((\Gamma_{r}')^{-1}(t))|\approx 1,
\end{align}
which further leads to $((\Gamma_{r}')^{-1})'(t)=1/\Gamma_{r}''((\Gamma_{r}')^{-1}(t)) \ls 1$. On the other hand, as Lemma \ref{lemma 2.1}(iv), we also have
\begin{align}\label{eq:20.21}
|\Gamma_{r}^{(k)}((\Gamma_{r}')^{-1}(t))|\ls 1\quad \textrm{for}~ \textrm{all}~ 2\leq k\leq N.
\end{align}
By simple calculation, it is easy to see that \eqref{eq:20.20} and \eqref{eq:20.21} imply  Lemma \ref{lemma 2.1}(v).
\end{proof}

\begin{lemma}\label{lemma 2.2}
Let $c\in C^M(\mathbb{R})$ supported on $\{t\in \mathbb{R}:\  |t|\ls 1\}$, $M\in\mathbb{N} $. Assume that $|c^{(k)}(t)|\ls |t|^{2m-k}$ for all $k\leq M$, where $m,k\in\mathbb{N} $ and $m+1\leq M$. Then, for all $\lambda>0$, we have
\begin{align*}
\left|\int_{-\infty}^{\infty} e^{-i\lambda t^{2}}c(t)\,\textrm{d}t\right|\lesssim \lambda^{-\frac{1}{2}-m}.
\end{align*}
\end{lemma}

\begin{proof}[Proof of Lemma \ref{lemma 2.2}.]  We will adopt a  similar argument to  that in \cite[page 335, Proposition 3]{S}.  Let $\varpi:\ \mathbb{R}\rightarrow\mathbb{R}$ be a smooth function supported on $\{t\in \mathbb{R}:\  |t|\leq 2\}$  such that $\varpi(t)= 1$ on $\{t\in \mathbb{R}:\  |t|\leq 1\}$, we then write
\begin{align*}
\int_{-\infty}^{\infty} e^{-i\lambda t^{2}}c(t)\,\textrm{d}t=\int_{-\infty}^{\infty} e^{-i\lambda t^{2}}c(t)\varpi(t/ \varepsilon)\,\textrm{d}t+\int_{-\infty}^{\infty} e^{-i\lambda t^{2}}c(t)(1-\varpi(t/ \varepsilon))\,\textrm{d}t,
\end{align*}
where $\varepsilon>0$ will be set later. The first integral is bounded by
\begin{align*}
\int_{|t|\leq 2\varepsilon} |t|^{2m}\,\textrm{d}t\ls \varepsilon^{2m+1}.
\end{align*}
The second integral can be written as
\begin{align*}
\int_{-\infty}^{\infty} e^{-i\lambda t^{2}}D^{(k)}\left(c(\cdot)(1-\varpi(\cdot/ \varepsilon))\right)(t)\,\textrm{d}t,
\end{align*}
where $D^{(k)}(f)(t):=\frac{1}{i2\lambda}\frac{\textrm{d}}{\textrm{d}t}(\frac{f(t)}{ t})$. By calculation, the second integral can be further bounded by
\begin{align*}
\lambda^{-k}\int_{|t|\geq\varepsilon} |t|^{2m-2k}\,\textrm{d}t=\lambda^{-k}\varepsilon^{2m-2k+1},\quad \textrm{if}~ 2m-2k+1<0.
\end{align*}
Letting $\varepsilon:=\lambda^{-1/2}$ with $k=m+1$ yields the desired estimate.
\end{proof}

\begin{lemma}\label{lemma 2.3}
Let
\begin{align*}
a(\lambda,t_0):=\int_{-\infty}^{\infty} e^{-i\left(\lambda \frac{ t^{2}}{2}\Gamma_{r}''(t_{0}) +\lambda \frac{t^{3}}{2}\int_{0}^{1}(1-\theta)^2\Gamma_{r}'''(\theta t+t_{0})\,\textrm{d}\theta\right)}\tilde{\phi}(t+t_0)\,\textrm{d}t,
\end{align*}
where $\tilde{\phi}(t):=\phi(t)\chi_{(-\infty,0)}(t)$. Then, for $|\lambda|\gs 1$ and $t_{0}<0$ with $|t_0|\approx 1$, we have
\begin{align*}
\left|\partial_{\lambda}^{\alpha}\partial_{t_{0}}^{\beta}a(\lambda,t_{0})\right|\lesssim |\lambda|^{-\frac{1}{2}-|\alpha|} \quad \textrm{for}~ \textrm{all} ~\alpha,\beta\in \mathbb{N}~ \textrm{and}~\beta<N-2,
\end{align*}
with a bound independent of $r$.
\end{lemma}

\begin{proof}[Proof of Lemma \ref{lemma 2.3}.] We first verify $|\partial_{\lambda}\partial_{t_{0}}a(\lambda,t_{0})|\lesssim |\lambda|^{-3/2}$.  A  computation gives
$$\partial_{\lambda}\partial_{t_{0}}a(\lambda,t_{0})=\int_{-\infty}^{\infty} e^{-i\left(\lambda \frac{ t^{2}}{2}\Gamma_{r}''(t_{0}) +\lambda \frac{t^{3}}{2}\int_{0}^{1}(1-\theta)^2\Gamma_{r}'''(\theta t+t_{0})\,\textrm{d}\theta\right)}c_{1,1}(\lambda,t,t_{0})\,\textrm{d}t,$$
where $c_{1,1}(\lambda,t,t_{0})$ can be expressed as the sum of the following four terms:
\begin{eqnarray*}
\left\{\aligned
&-\left(\lambda\frac{t^{2}}{2}\Gamma_{r}'''(t_{0}) \right)\left(\frac{t^{2}}{2}\Gamma_{r}''(t_{0})+\frac{t^{3}}{2}\int_{0}^{1}(1-\theta)^2\Gamma_{r}'''(\theta t+t_{0})\,\textrm{d}\theta\right)\tilde{\phi}(t+t_{0});\\
&-\left(\lambda\frac{t^{3}}{2} \int_{0}^{1}(1-\theta)^2\Gamma_{r}^{(4)}(\theta t+t_{0})\,\textrm{d}\theta\right)\left(\frac{t^{2}}{2}\Gamma_{r}''(t_{0})+\frac{t^{3}}{2}\int_{0}^{1}(1-\theta)^2\Gamma_{r}'''(\theta t+t_{0})\,\textrm{d}\theta\right)\tilde{\phi}(t+t_{0});\\
&-i\left(\frac{t^{2}}{2}\Gamma_{r}'''(t_{0})+\frac{t^{3}}{2}\int_{0}^{1}(1-\theta)^2\Gamma_{r}^{(4)}(\theta t+t_{0})\,\textrm{d}\theta\right)\tilde{\phi}(t+t_{0});\\
&-i\left(\frac{t^{2}}{2}\Gamma_{r}''(t_{0})+\frac{t^{3}}{2}\int_{0}^{1}(1-\theta)^2\Gamma_{r}'''(\theta t+t_{0})\,\textrm{d}\theta\right)\tilde{\phi}'(t+t_{0}).
\endaligned\right.
\end{eqnarray*}

The fact that $t_{0}<0$ with $|t_0|\approx 1$ and $t+t_{0}<0$ with $|t+t_0|\approx 1$  imply that $|t|\ls 1$, $\theta t+t_{0}=\theta (t+t_{0})+(1-\theta) t_{0}<0$ and $|\theta t+t_{0}|\approx 1$. By Lemma \ref{lemma 2.1}, we  have $|\Gamma_{r}^{(n)}(\theta t+t_{0})|\ls 1$ for all $ n\leq N$ and $|\Gamma_{r}^{(2)}(\theta t+t_{0})|\approx 1$. By simple calculation, one has
\begin{align*}
\left|\partial_{t}^{k}e^{-i\lambda \frac{t^{3}}{2}\int_{0}^{1}(1-\theta)^2\Gamma_{r}'''(\theta t+t_{0})\,\textrm{d}\theta} c_{1,1}(\lambda,t,t_{0})\right|\lesssim |\lambda| |t|^{4-k}+|t|^{2-k} \quad \textrm{for}~ \textrm{all} ~k\in \mathbb{N}~ \textrm{with}~k<N-3.
\end{align*}
Then, by Lemma \ref{lemma 2.2}, we get
\begin{align*}
|\partial_{\lambda}\partial_{t_{0}}a(\lambda,t_{0})|\lesssim |\lambda||\lambda\Gamma_{r}''(t_{0})|^{-\frac{1}{2}-2}+|\lambda\Gamma_{r}''(t_{0})|^{-\frac{1}{2}-1}\lesssim |\lambda|^{-\frac{1}{2}-1}.
\end{align*}

Similarly, one can  write
$$\partial_{\lambda}^{\alpha}\partial_{t_{0}}^{\beta}a(\lambda,t_{0})=\int_{-\infty}^{\infty} e^{-i\left(\lambda \frac{ t^{2}}{2}\Gamma_{r}''(t_{0}) +\lambda \frac{t^{3}}{2}\int_{0}^{1}(1-\theta)^2\Gamma_{r}'''(\theta t+t_{0})\,\textrm{d}\theta\right)}c_{\alpha,\beta}(\lambda,t,t_{0})\,\textrm{d}t,$$
where
\begin{align*}
\left|\partial_{t}^{k}e^{-i\lambda \frac{t^{3}}{2}\int_{0}^{1}(1-\theta)^2\Gamma_{r}'''(\theta t+t_{0})\,\textrm{d}\theta}c_{\alpha,\beta}(\lambda,t,t_{0})\right|\lesssim |\lambda|^{\beta}|t|^{2\alpha+2\beta-k}+|t|^{2\alpha-k} \quad \textrm{for}~ \textrm{all} ~k\in \mathbb{N}~ \textrm{with}~k<N-2-\beta.
\end{align*}
By Lemma \ref{lemma 2.2}, for $N$ large enough, we have
$$\left|\partial_{\lambda}^{\alpha}\partial_{t_{0}}^{\beta}a(\lambda,t_{0})\right|\lesssim |\lambda|^{-\frac{1}{2}-\alpha} \quad \textrm{for}~ \textrm{all} ~\alpha,\beta\in \mathbb{N}~ \textrm{with}~\beta<N-2.$$
This finishes the proof of Lemma \ref{lemma 2.3}.
\end{proof}

\bigskip

In the following, we will prove \eqref{eq:20.16}.

\noindent {\bf Proof of \eqref{eq:20.16}.} \ Let $\Omega(u)\in C_{c}^{\infty}(1/2,5/2)$ and $\Omega(u)=1$ if $u\in[1,2]$, and define
\begin{align*}
\tilde{T}_{u}f(x_1,x_2):=\Omega(u)\int_{-\infty}^{\infty} f(x_1-t,x_2-u\Gamma_{r}(t))\phi(t)\,\textrm{d}t.
\end{align*}
It suffices to prove
\begin{align*}
\Bigg(\int_{\frac{1}{2}}^{\frac{5}{2}} \left\|\tilde{T}_uP_\kappa f\right\|^p_{L^{p}(\mathbb{R}^{2})}\,\textrm{d}u\Bigg)^{\frac{1}{p}}\ls 2^{-\left(\delta+\frac{1}{p}\right) \kappa}\|f\|_{L^{p}(\mathbb{R}^{2})}.
\end{align*}
By taking Fourier transform, we have
\begin{align*}
\tilde{T}_{u}P_\kappa f(x_1,x_2)=\int_{-\infty}^{\infty} \int_{-\infty}^{\infty}  e^{ix_1\xi+ix_2\eta}m(u,\xi,\eta)\hat{f}(\xi,\eta)\,\textrm{d}\xi \,\textrm{d}\eta,
\end{align*}
where the multiplier of $\tilde{T}_uP_\kappa f$ is defined as
\begin{align*}
m(u,\xi,\eta):=\Omega(u)\psi\left(2^{-\kappa }(\xi^2+\eta^2)^{\frac{1}{2}}\right)\int_{-\infty}^{\infty} e^{-i(t\xi+u\Gamma_{r}(t)\eta)}\phi(t)\,\textrm{d}t.
\end{align*}
 Denote the corresponding phase function by
\begin{align*}
\Phi_u(t):=t\xi+u\Gamma_{r}(t)\eta,
\end{align*}
and we  have
\begin{align*}
\Phi'_u(t)=\xi+u\Gamma'_{r}(t)\eta \quad \textrm{and} \quad \Phi''_u(t)=u\Gamma''_{r}(t)\eta.
\end{align*}

By (ii) of Lemma \ref{lemma 2.1}, we get $C^{(1)}_2/2e^{C^{(1)}_{3}}\leq\Gamma_{r}'(t)\leq 2e^{C^{(1)}_{3}}C^{(1)}_3$. Then, if $|\xi|\geq 6e^{C^{(1)}_{3}}C^{(1)}_3|\eta|$, it follows that $|\Phi'_u(t)|\geq |\xi|-|u\Gamma'_{r}(t)\eta|\gtrsim |\xi|+|\eta|$.
Similarly, if $|\eta|\geq (6e^{C^{(1)}_{3}}/C^{(1)}_2) |\xi|$, we obtain $|\Phi'_u(t)|\gtrsim |\xi|+|\eta|$. Integration by parts yields
\begin{align*}
\left|\int_{-\infty}^{\infty} e^{-i(t\xi+u\Gamma_{r}(t)\eta)}\phi(t)\,\textrm{d}t\right| \ls (|\xi|+|\eta|)^{-n} \quad \textrm{for}~  n\leq 4.
\end{align*}
Furthermore, noting that $2^\kappa\approx \sqrt{\xi^2+\eta^2}\approx \max\{|\xi|,|\eta|\}$ and $\kappa\in \mathbb{N}$, one can obtain
\begin{align}\label{eq:20.23}
\left|\partial_{\xi}^{\alpha}\partial_{\eta}^{\beta}\bigg((1-\chi(|\xi|/|\eta|))m(u,\xi,\eta)\bigg)\right|\lesssim 2^{-\kappa}(1+|\xi|+|\eta|)^{-3}\quad \textrm{for}~ \textrm{all}~ \alpha,\beta\in\mathbb{N} ~\textrm{and}~ \alpha+\beta\leq3,
\end{align}
where $\chi \in C_{c}^{\infty}(\mathbb{R}^+)$ such that $\chi=1$ on $[C^{(1)}_2/6e^{C^{(1)}_{3}},6e^{C^{(1)}_{3}}C^{(1)}_3]$.

Denote
\begin{align}\label{eq:20.24}
\tilde{T}^1_{u}P_\kappa f(x_1,x_2):=\int_{-\infty}^{\infty} \int_{-\infty}^{\infty}  e^{ix_1\xi+ix_2\eta}\chi(|\xi|/|\eta|)m(u,\xi,\eta)\hat{f}(\xi,\eta)\,\textrm{d}\xi \,\textrm{d}\eta
\end{align}
and
\begin{align}\label{eq:20.25}
\tilde{T}^2_{u}P_\kappa f(x_1,x_2):=\int_{-\infty}^{\infty} \int_{-\infty}^{\infty}  e^{ix_1\xi+ix_2\eta}(1-\chi(|\xi|/|\eta|))m(u,\xi,\eta)\hat{f}(\xi,\eta)\,\textrm{d}\xi \,\textrm{d}\eta.
\end{align}
Since\footnote{We denote $\mathcal{F}_{\xi,\eta}$ means the \emph{Fourier transform} about $(\xi,\eta)\in \mathbb{R}^2$.} $\mathcal{F}_{\xi,\eta}(2^\kappa(1-\chi(|\xi|/|\eta|))m(u,\xi,\eta))\in L^{1}(\mathbb{R}^2)$ by \eqref{eq:20.23}, then
$$\|\tilde{T}^2_uP_\kappa f\|_{L^{p}(\mathbb{R}^{2})}\ls 2^{-\kappa}\| f\|_{L^{p}(\mathbb{R}^{2})},
$$
which yields
\begin{align*}
\Bigg(\int_{\frac{1}{2}}^{\frac{5}{2}} \left\|\tilde{T}^2_uP_\kappa f\right\|^p_{L^{p}(\mathbb{R}^{2})}\,\textrm{d}u\Bigg)^{\frac{1}{p}}\ls 2^{-\left(\delta+\frac{1}{p}\right) \kappa}\|f\|_{L^{p}(\mathbb{R}^{2})}\quad \textrm{for}~ \textrm{all} ~\kappa\in \mathbb{N}~ \textrm{and}~p\in (2,\infty).
\end{align*}
Thus, it remains to prove
\begin{align}\label{eq:20.27}
\Bigg(\int_{\frac{1}{2}}^{\frac{5}{2}} \left\|\tilde{T}^1_uP_\kappa f\right\|^p_{L^{p}(\mathbb{R}^{2})}\,\textrm{d}u\Bigg)^{\frac{1}{p}}\ls 2^{-\left(\delta+\frac{1}{p}\right) \kappa}\|f\|_{L^{p}(\mathbb{R}^{2})}\quad \textrm{for}~ \textrm{all} ~\kappa\in \mathbb{N}~ \textrm{and}~p\in (2,\infty).
\end{align}
Without loss of generality, we may assume that $\xi,\eta>0$ in \eqref{eq:20.24}, and the other cases can be treated similarly. If $\gamma$ is odd, then  $\Gamma_{r}$ is also  odd. Furthermore, we have $\Phi'_u(t)\gtrsim \xi+\eta$. Then a similar argument to that of  $\tilde{T}^2_uP_\kappa f$ gives \eqref{eq:20.27} in this case. If $\gamma$ is an even function, it implies that $\Gamma_{r}$ is also an even function.  We  will consider the following two cases. Define
\begin{eqnarray*}
\left\{\aligned
\tilde{m}^1_-(u,\xi,\eta):=\chi(|\xi|/|\eta|)\Omega(u)\psi\left(2^{-\kappa }(\xi^2+\eta^2)^{\frac{1}{2}}\right)\left(\int_{-\infty}^{0} e^{-i(t\xi+u\Gamma_{r}(t)\eta)}\phi(t)\,\textrm{d}t\right)\mathbf{1}_{\mathbb{R}^+}(\xi)\mathbf{1}_{\mathbb{R}^+}(\eta),\\
\tilde{m}^1_+(u,\xi,\eta):=\chi(|\xi|/|\eta|)\Omega(u)\psi\left(2^{-\kappa }(\xi^2+\eta^2)^{\frac{1}{2}}\right)\left(\int_{0}^{\infty} e^{-i(t\xi+u\Gamma_{r}(t)\eta)}\phi(t)\,\textrm{d}t\right)\mathbf{1}_{\mathbb{R}^+}(\xi)\mathbf{1}_{\mathbb{R}^+}(\eta),
\endaligned\right.
\end{eqnarray*}
and the associated operators are $\tilde{T}_{u,-}^{1}P_\kappa$ and $\tilde{T}_{u,+}^{1}P_\kappa$, respectively.

For $\tilde{T}_{u,+}^{1}P_\kappa$, it is easy to see that $\Phi'_u(t)\gtrsim \xi+\eta$ in $\tilde{m}^1_+$. Therefore, as $\tilde{T}^2_uP_\kappa f$, we may also obtain
\begin{align*}
\Bigg(\int_{\frac{1}{2}}^{\frac{5}{2}} \left\|\tilde{T}_{u,+}^{1}P_\kappa f\right\|^p_{L^{p}(\mathbb{R}^{2})}\,\textrm{d}u\Bigg)^{\frac{1}{p}}\ls 2^{-\left(\delta+\frac{1}{p}\right) \kappa}\|f\|_{L^{p}(\mathbb{R}^{2})}\quad \textrm{for}~ \textrm{all} ~\kappa\in \mathbb{N}~ \textrm{and}~p\in (2,\infty).
\end{align*}

For $\tilde{T}_{u,-}^{1}P_\kappa$, let $\Phi'_u(t_0)=\xi+u\Gamma'_{r}(t_0)\eta=0 $, then $t_{0}<0$ is the critical point. Noting that $\gamma'$ is strictly monotonic on $\mathbb{R}^+$ with $\gamma'(\mathbb{R}^+)=\mathbb{R}^+$, we may write $t_{0}:=(\Gamma_{r}')^{-1}(-\xi/u\eta)$. By Taylor's formula, we have
\begin{align*}
\Phi_u(t+t_0)=t_{0}\xi+u\Gamma_{r}(t_{0})\eta+\frac{u \eta}{2}t^{2} \Gamma_{r}''(t_{0}) +\frac{u \eta }{2}t^{3}\int_{0}^{1}(1-\theta)^2\Gamma_{r}'''(\theta t+t_{0})\,\textrm{d}\theta.
\end{align*}
Denote $\tilde{\phi}(t):=\phi(t)\mathbf{1}_{(-\infty,0)}(t)$. One can  write
\begin{align*}
\tilde{m}^1_-(u,\xi,\eta)=e^{i\Phi(u,\xi,\eta)}c(u,\xi,\eta),
\end{align*}
where
\begin{align*}
\Phi(u,\xi,\eta):=-\Phi_u(t_0)=-t_{0}\xi-u\Gamma_{r}(t_{0})\eta;
\end{align*}
\begin{align*}
c(u,\xi,\eta):=\chi(|\xi|/|\eta|)\Omega(u)\psi\left(2^{-\kappa }(\xi^2+\eta^2)^{\frac{1}{2}}\right)a(u\eta,t_0)\mathbf{1}_{\mathbb{R}^+}(\xi)\mathbf{1}_{\mathbb{R}^+}(\eta)
\end{align*}
and
\begin{align*}
a(\lambda,t_0):=\int_{-\infty}^{\infty} e^{-i\left(\lambda \frac{ t^{2}}{2}\Gamma_{r}''(t_{0}) +\lambda \frac{t^{3}}{2}\int_{0}^{1}(1-\theta)^2\Gamma_{r}'''(\theta t+t_{0})\,\textrm{d}\theta\right)}\tilde{\phi}(t+t_0)\,\textrm{d}t.
\end{align*}
As a result,
\begin{align*}
\tilde{T}_{u,-}^{1}P_\kappa f(x_1,x_2)=\int_{-\infty}^{\infty} \int_{-\infty}^{\infty}  e^{ix_1\xi+ix_2\eta}e^{i\Phi(u,\xi,\eta)}c(u,\xi,\eta)\hat{f}(\xi,\eta)\,\textrm{d}\xi \,\textrm{d}\eta,
\end{align*}
and it suffices to prove that there exists a positive constant $\delta$, independent of $r$, such that
\begin{align}\label{eq:20.33}
\Bigg(\int_{\frac{1}{2}}^{\frac{5}{2}} \left\|\tilde{T}_{u,-}^{1}P_\kappa f\right\|^p_{L^{p}(\mathbb{R}^{2})}\,\textrm{d}u\Bigg)^{\frac{1}{p}}\ls 2^{-\left(\delta+\frac{1}{p}\right) \kappa}\|f\|_{L^{p}(\mathbb{R}^{2})}\quad \textrm{for}~ \textrm{all} ~\kappa\in \mathbb{N}~ \textrm{and}~p\in (2,\infty).
\end{align}

Let $\varphi:\ \mathbb{R}^2\rightarrow\mathbb{R}$ be a smooth function supported on $\{(x_1,x_2)\in \mathbb{R}^2:\ |(x_1,x_2)|\leq 2\}$ such that $\varphi(x_1,x_2)= 1$ on $\{(x_1,x_2)\in \mathbb{R}^2:\ |(x_1,x_2)|\leq 1\}$. We write
\begin{align*}
b(x_1,x_2,u,\xi,\eta):=\varphi(x_1,x_2)c(u,\xi,\eta)\quad \textrm{and} \quad\Psi(x_1,x_2,u,\xi,\eta):=x_1\xi+x_2\eta+\Phi(u,\xi,\eta).
\end{align*}
Define
\begin{align*}
Af(x_1,x_2,u):=\int_{-\infty}^{\infty} \int_{-\infty}^{\infty}  e^{i\Psi(x_1,x_2,u,\xi,\eta)}b(x_1,x_2,u,\xi,\eta)\hat{f}(\xi,\eta)\,\textrm{d}\xi \,\textrm{d}\eta.
\end{align*}
We first prove that there exists a positive constant $\delta$ independent of $r$ such that
\begin{align}\label{eq:20.34}
\|Af\|_{L^{p}(\mathbb{R}^{3})}\ls 2^{-\left(\delta+\frac{1}{p}\right) \kappa}\|f\|_{L^{p}(\mathbb{R}^{2})}\quad \textrm{for}~ \textrm{all} ~\kappa\in \mathbb{N}~ \textrm{and}~p\in (2,\infty).
\end{align}
Recall that $t_0(u,\xi,\eta)=(\Gamma_{r}')^{-1}(-\xi/u\eta)$.  By (v) of  Lemma \ref{lemma 2.1}, we have
\begin{align}\label{eq:20.35}
\left|\partial_{u}^{\alpha}\partial_{\xi}^{\beta}\partial_{\eta}^{\gamma}t_{0}(u,\xi,\eta)\right|\lesssim (1+|\xi|+|\eta|)^{-\beta-\gamma} \quad \textrm{for}~ \textrm{all} ~\alpha,\beta,\gamma\in \mathbb{N}~ \textrm{with}~\alpha+\beta+\gamma<N.
\end{align}
By Lemma \ref{lemma 2.3}, for $|\lambda|\gs 1$ and $t_{0}<0$ with $|t_0|\approx 1$, we obtain
\begin{align*}
\left|\partial_{\lambda}^{\alpha}\partial_{t_{0}}^{\beta}a(\lambda,t_{0})\right|\lesssim |\lambda|^{-\frac{1}{2}-\alpha} \quad \textrm{for}~ \textrm{all} ~\alpha,\beta\in \mathbb{N}~ \textrm{with}~\beta<N-2.
\end{align*}
This, combined with \eqref{eq:20.35}, yields
\begin{align}\label{eq:20.37}
\left|\partial_{u}^{\alpha}\partial_{\xi}^{\beta}\partial_{\eta}^{\gamma}c(u,\xi,\eta)\right|\lesssim (1+|\xi|+|\eta|)^{-\frac{1}{2}-\beta-\gamma} \quad \textrm{for}~ \textrm{all} ~\alpha,\beta,\gamma\in \mathbb{N}~ \textrm{with}~\alpha+\beta+\gamma<N,
\end{align}
which implies that
\begin{align}\label{eq:20.370}
\left|\partial_{x_1}^{\alpha_{1}}\partial_{x_2}^{\alpha_{2}}\partial_{u}^{\alpha_{3}}\partial_{\xi}^{\alpha_{4}}\partial_{\eta}^{\alpha_{5}}b(x_1,x_2,u,\xi,\eta)\right|\leq(1+|\xi|+|\eta|)^{-\frac{1}{2}-\alpha_{4}-\alpha_{5}}
\end{align}
for all $\alpha_{1},\cdots,\alpha_{5}\in \mathbb{N}$ with $\alpha_{3}+\alpha_{4}+\alpha_{5}<N$. It is easy to see that $\textrm{supp} ~b \subset \{(x_1,x_2)\in \mathbb{R}^2:\ |(x_1,x_2)|\leq 2\}\times(1/2,5/2)\times\{(\xi,\eta)\in \mathbb{R}^2:\ \xi\approx\eta\approx2^{k},\ \xi>0,\ \eta>0\}$.

\eqref{eq:20.37} implies the multiplier of $\tilde{T}_{u,-}^{1}P_\kappa$ can be bounded from above by $2^{-\kappa/2}$. We apply Plancherel's theorem to obtain
\begin{align*}
\Bigg(\int_{\frac{1}{2}}^{\frac{5}{2}}\left\|\tilde{T}_{u,-}^{1}P_\kappa f\right\|^2_{L^{2}(\mathbb{R}^{2})}\,\textrm{d}u\Bigg)^{\frac{1}{2}}\ls 2^{-\frac{\kappa}{2}}\|f\|_{L^{2}(\mathbb{R}^{2})}.
\end{align*}
By interpolation, it suffices to prove \eqref{eq:20.33} for $p\in [6,\infty)$. We will first prove \eqref{eq:20.34} for $p\in [6,\infty)$ by using  local smoothing estimates provided in \cite{B}, then prove   \eqref{eq:20.33}  by using  \eqref{eq:20.34} in the case $p\in [6,\infty)$.

For the convenience of writing, we denote $e_1:=(1,0), B(0,2):=\{(x_1,x_2)\in \mathbb{R}^2:\ |(x_1,x_2)|\leq 2\}$, $Z:=B(0,2)\times(1/2,5/2)$ and
$$\Gamma:=\left\{Q^{-1}(\xi,\eta)\in \mathbb{R}^2:\ 1/4\leq\xi^2+\eta^2\leq4, C^{(1)}_2/6e^{C^{(1)}_{3}}\leq |\xi|/|\eta|  \leq 6e^{C^{(1)}_{3}}C^{(1)}_3, \xi>0, \eta>0 \right\},$$
where $Q$ is the orthogonal matrix so that $Qe_{1}=(\sqrt{2}/ 2,\sqrt{2}/ 2)$. We write $\Phi_Q(u,\xi,\eta):=\Phi(u,Q(\xi,\eta))$, $b_Q(x_1,x_2,u,\xi,\eta):=b(x_1,x_2,u,Q(\xi,\eta))$ and $\Psi_Q(x_1,x_2,u,\xi,\eta):=\Psi(x_1,x_2,u,Q(\xi,\eta))$. For $\alpha,\beta,\gamma\in \mathbb{N}$ with $\alpha+\beta+\gamma<N$ and $\alpha_{1},\cdots,\alpha_{5}\in \mathbb{N}$ with $\alpha_{3}+\alpha_{4}+\alpha_{5}<N$, we now claim the following four terms hold:
\begin{align}\label{eq:20.38}
\left|\partial_{u}^{\alpha}\partial_{\xi}^{\beta}\partial_{\eta}^{\gamma}\Phi_Q(u,\xi,\eta)\right|\ls 1 \quad \textrm{on} ~ Z\times\Gamma;
\end{align}
\begin{align}\label{eq:20.380}
\left|\partial_{x_1}^{\alpha_{1}}\partial_{x_2}^{\alpha_{2}}\partial_{u}^{\alpha_{3}}\partial_{\xi}^{\alpha_{4}}\partial_{\eta}^{\alpha_{5}}b_Q(x_1,x_2,u,\xi,\eta)\right|\ls (1+|\xi|+|\eta|)^{-\frac{1}{2}-\alpha_{4}-\alpha_{5}} \quad \textrm{on} ~ Z\times\Gamma;
\end{align}
\begin{align}\label{eq:20.39}
\left(
\begin{array}{ccc}
\partial_{x_1 \xi} \Psi_Q(x_1,x_2,u,\xi,\eta)~& ~ \partial_{x_2 \xi}\Psi_Q(x_1,x_2,u,\xi,\eta)  \\
 \partial_{x_1 \eta} \Psi_Q(x_1,x_2,u,\xi,\eta)~ &~ \partial_{x_2 \eta} \Psi_Q(x_1,x_2,u,\xi,\eta)
\end{array}
\right)=Q^T \quad \textrm{on} ~ Z\times\Gamma;
\end{align}
\noindent and
\begin{align}\label{eq:20.40}
\left|\partial_{\eta}^{2}\partial_{u}\Psi_Q(x_1,x_2,u,\xi,\eta)\right|\gs 1 \quad \textrm{on}~ Z\times\Gamma,
\end{align}
where the bounds are independent of $r$.

%Note that the conditions \eqref{eq:20.39}, \eqref{eq:20.40} are the ones $H1')$, $H2')$ in \cite[Page %7]{B}, which are stronger than $H1)$, $H2)$ in \cite[Page 4]{B}, respectively.

Recall that $\Psi(x_1,x_2,u,\xi,\eta)=x_1\xi+x_2\eta+\Phi(u,\xi,\eta)$, $\Phi(u,\xi,\eta)=-t_{0}\xi-u\Gamma_{r}(t_{0})\eta$ and $t_{0}=(\Gamma_{r}')^{-1}(-\xi/u\eta)$. Then, \eqref{eq:20.38}  follows from \eqref{eq:20.35} and Lemma \ref{lemma 2.1}; \eqref{eq:20.380}  follows from \eqref{eq:20.370} and Lemma \ref{lemma 2.1}; \eqref{eq:20.39} is easy to check. We now turn to verify \eqref{eq:20.40}.  It suffices to verify
\begin{align*}
\left|\partial_{\xi}^{2}\partial_{u}\Phi(u,\xi,\eta)-2\partial_{\xi}\partial_{\eta}\partial_{u}\Phi(u,\xi,\eta)+\partial_{\eta}^{2}\partial_{u}\Phi(u,\xi,\eta)\right|\gs 1 \quad \textrm{on}~ Z\times Q\Gamma.
\end{align*}
Indeed, by noting that $\xi+u\Gamma_{r}'(t_{0})\eta=0$, we have the following results:
$$\partial_{u}t_{0}=\frac{\xi}{u^{2}\eta\Gamma_{r}''(t_0)}; \quad  \partial_{\xi}t_{0}=-\frac{1}{u\eta\Gamma_{r}''(t_0)}; \quad  \partial_{\eta}t_{0}=\frac{\xi}{u\eta^{2}\Gamma_{r}''(t_0)};$$
$$ \partial_{u}\Phi(u,\xi,\eta)=-\Gamma_{r}(t_0)\eta; \quad \partial_{\xi}\partial_{u}\Phi(u,\xi,\eta)=\frac{\Gamma_{r}'(t_0)}{u\Gamma_{r}''(t_0)}; \quad
\partial_{\eta}\partial_{u}\Phi(u,\xi,\eta)
=-\frac{\xi\Gamma_{r}'(t_0)}{u\eta\Gamma_{r}''(t_0)}-\Gamma_{r}(t_0); $$
$$\partial_{\xi}^{2}\partial_{u}\Phi(u,\xi,\eta)
=\frac{\Gamma_{r}'(t_0)\Gamma_{r}'''(t_0)-\Gamma_{r}''^{2}(t_0)}{u^{2}\eta\Gamma_{r}''(t_0)^{3}}; \quad \partial_{\xi}\partial_{\eta}\partial_{u}\Phi(u,\xi,\eta)
=-\frac{\Gamma_{r}'(t_0)\Gamma_{r}'''(t_0)-\Gamma_{r}''(t_0)^{2}}{u^{2}\eta^{2}\Gamma_{r}''(t_0)^{3}}\xi;\quad
$$
$$\partial_{\eta}^{2}\partial_{u}\Phi(u,\xi,\eta)
=\frac{\Gamma_{r}'\Gamma_{r}'''(t_0)-\Gamma_{r}''(t_0)^{2}}{u^{2}\eta^{3}\Gamma_{r}''(t_0)^{3}}\xi^{2}.$$
 Note that $\xi\approx\eta\approx1$ and $\xi>0$ and $\eta>0$ if $(\xi,\eta)\in Q\Gamma$. By Lemma \ref{lemma 2.1} and (i) of {\bf(H.)}, we then have
\begin{align*}
&\left|\partial_{\xi}^{2}\partial_{u}\Phi(u,\xi,\eta)-2\partial_{\xi}\partial_{\eta}\partial_{u}\Phi(u,\xi,\eta)+\partial_{\eta}^{2}\partial_{u}
\Phi(u,\xi,\eta)\right|\\
=&\left|\frac{\Gamma_{r}'(t_0)\Gamma_{r}'''(t_0)-\Gamma_{r}''(t_0)^{2}}{u^{2}\eta^{3}\Gamma_{r}''(t_0)^{3}}\right| |\xi+\eta|^{2}\gs 1\qquad \qquad \qquad \qquad \textrm{on}~ Z\times Q\Gamma.
\end{align*}

 Combining \eqref{eq:20.38}-\eqref{eq:20.40}, we can apply \cite[Proposition 3.2]{B} to obtain, for any $ p\in[6,\infty)$ and $\kappa\in \mathbb{N}$, there exists $\delta\in(0,1/p)$ independent of $r$ such that
\begin{align}\label{eq:20.41}
\|Af\|_{L^{p}(\mathbb{R}^{3})}= 2^{-\left(\delta+\frac{1}{p}\right) \kappa}\left\|2^{\left(\delta+\frac{1}{p}\right) \kappa}Af\right\|_{L^{p}(\mathbb{R}^{3})} \ls 2^{-\left(\delta+\frac{1}{p}\right) \kappa}\|f\|_{L^{p}(\mathbb{R}^{2})}
\end{align}
with a bound  independent of $r$. This completes the proof of \eqref{eq:20.34} for $p\in [6,\infty)$.

We now will  use \eqref{eq:20.34} to prove \eqref{eq:20.33} for any $p\in [6,\infty)$. Note that
$$\varphi(x_1,x_2) \tilde{T}_{u,-}^{1}P_{\kappa}f(x_1,x_2)=\int_{\mathbb{R}^4} e^{-ix\xi}e^{-iy\eta}e^{ix_1\xi}e^{ix_2\eta}e^{i\Phi(u,\xi,\eta)}\varphi(x_1,x_2)c(u,\xi,\eta)f(x,y)\,\textrm{d}x\,\textrm{d}y\,\textrm{d}\xi \,\textrm{d}\eta.$$
Denote $L_{\xi}f:=(1-\Delta_{\xi})(1+|x|^{2})^{-1}$ and $L_{\eta}f:=(1-\Delta_{\eta})(1+|y|^{2})^{-1}$. One can write
$$L_{\xi}L_{\eta}\left(e^{ix_1\xi}e^{ix_2\eta}e^{i\Phi(u,\xi,\eta)}\varphi(x_1,x_2)c(u,\xi,\eta)\right)=e^{ix_1\xi}e^{ix_2\eta}
e^{i\Phi(u,\xi,\eta)}\tilde{c}(x_1,x_2,u,\xi,\eta)(1+|x|^{2})^{-1}(1+|y|^{2})^{-1}$$ for some $\tilde{c}\in S^{-1/2}$ satisfying
$$ \left|\partial_{x_1}^{\alpha_{1}}\partial_{x_2}^{\alpha_{2}}\partial_{u}^{\alpha_{3}}\partial_{\xi}^{\alpha_{4}}\partial_{\eta}^{\alpha_{5}}\tilde{c}(x_1,x_2,u,\xi,\eta)\right|\ls (1+|\xi|+|\eta)^{-\frac{1}{2}-\alpha_{4}-\alpha_{5}}$$
for all $\alpha_{1},\cdots,\alpha_{5}\in \mathbb{N}$ with $\alpha_{3}+\alpha_{4}+\alpha_{5}<N$, where the constant is independent of $r$, and
$\textrm{supp}~\tilde{c}\subset B(0,2)\times(1/2,5/2)\times\{(\xi,\eta)\in \mathbb{R}^2:\ \xi\approx\eta\approx2^{k},\ \xi>0,\ \eta>0\}$.

Integration by parts shows that $\varphi(x_1,x_2) \tilde{T}_{u,-}^{1}P_{\kappa}f(x_1,x_2)$ can be written as
\begin{align*}
&\int_{\mathbb{R}^4} e^{-ix\xi}e^{-iy\eta}e^{ix_1\xi}e^{ix_2\eta}e^{i\Phi(u,\xi,\eta)}\tilde{c}(x_1,x_2,u,\xi,\eta)(1+|x|^{2})^{-1}(1+|y|^{2})^{-1}f(x,y)\,\textrm{d}x\,\textrm{d}y\,\textrm{d}\xi \,\textrm{d}\eta\\
=& \int_{\mathbb{R}^2}  e^{ix_1\xi}e^{ix_2\eta}e^{i\Phi(u,\xi,\eta)}\tilde{c}(x_1,x_2,u,\xi,\eta) \hat{g}(\xi,\eta)  \,\textrm{d}\xi \,\textrm{d}\eta,
\end{align*}
where $g(x,y):=f(x,y)(1+|x|^{2})^{-1}(1+|y|^{2})^{-1}$. As  in \eqref{eq:20.41}, we get
\begin{align*}
\left\|\varphi \tilde{T}_{u,-}^{1}P_{\kappa}f\right\|^p_{L^{p}(\mathbb{R}^{3})}= 2^{-\left(\delta+\frac{1}{p}\right)p \kappa} \left\|g\right\|^p_{L^{p}(\mathbb{R}^{3})} \ls 2^{-\left(\delta+\frac{1}{p}\right)p \kappa} \int_{\mathbb{R}^2} \frac{|f(x,y)|^p}{(1+|x|^{2}+|y|^{2})^p} \,\textrm{d}x\,\textrm{d}y.
\end{align*}
Then, because $\tilde{T}_{u,-}^{1}P_{\kappa}f(x_1+x_{0},x_2+y_{0})=\tilde{T}_{u,-}^{1}P_{\kappa}(f(\cdot+x_{0},\cdot+y_{0}))(x_1,x_2)$, we obtain
\begin{align*}
&\int_{\mathbb{R}^3}|\varphi(x_1+x_{0},x_2+y_{0}) \tilde{T}_{u,-}^{1}P_{\kappa}f(x_1,x_2,u)|^{p}\,\textrm{d}x_1\,\textrm{d}x_2\,\textrm{d}u\\
\ls~ &2^{-\left(\delta+\frac{1}{p}\right)p \kappa}\int_{\mathbb{R}^2} \frac{|f(x,y)|^{p}}{(1+|x-x_{0}|^2+|y-y_{0}|^2)^{p}}\,\textrm{d}x\,\textrm{d}y.
\end{align*}
By integration in $x_{0},y_{0}$, we have
\begin{align*}
\int_{\mathbb{R}^3}| \tilde{T}_{u,-}^{1}P_{\kappa}f(x_1,x_2,u)|^{p}\,\textrm{d}x_1\,\textrm{d}x_2\,\textrm{d}u
\lesssim 2^{-\left(\delta+\frac{1}{p}\right)p \kappa}\int_{\mathbb{R}^2} |f(x,y)|^{p}\,\textrm{d}x\,\textrm{d}y.
\end{align*}
We have proved  \eqref{eq:20.33} for any $p\in [6,\infty)$  and thus finish the proof of \eqref{eq:20.16}.

\bigskip

\noindent {\bf Proof of \eqref{eq:20.17}:} The proof of \eqref{eq:20.17} is very similar as that of \eqref{eq:20.16} with two slight modifications.  Firstly, we will  replace $\phi(t)$ in \eqref{eq:20.a} by $-i\Gamma_{r}(t)\phi(t)$. Note that the latter maybe belongs to $C^{N}$ on $\{t\in \mathbb{R}:\ 1/2\leq |t|\leq 2\}$ with $N\in\mathbb{N}$ large enough, but it doesn't affect our proof. Secondly, we write $\eta\psi(2^{-\kappa }(\xi^2+\eta^2)^{\frac{1}{2}})$ in \eqref{eq:20.b} as

 $$
 2^\kappa\times \left(2^{-\kappa }(\xi^2+\eta^2)^{\frac{1}{2}}\psi(2^{-\kappa }(\xi^2+\eta^2)^{\frac{1}{2}})\right)\times \left(\frac{\eta }{(\xi^2+\eta^2)^{\frac{1}{2}}}\right).
 $$
Note that $2^{-\kappa }(\xi^2+\eta^2)^{\frac{1}{2}}\psi(2^{-\kappa }(\xi^2+\eta^2)^{\frac{1}{2}})$ just like  $\psi(2^{-\kappa }(\xi^2+\eta^2)^{\frac{1}{2}})$, and $\eta /(\xi^2+\eta^2)^{\frac{1}{2}}$ is  a harmless quantity. By repeating the process of proving \eqref{eq:20.16}, we can control the LHS of \eqref{eq:20.17} by $2^{-\left(\delta+\frac{1}{p}\right) \kappa}\|f\|_{L^{p}(\mathbb{R}^{2})}$ times a extra factor  $2^{\kappa}$, which is just the RHS of  \eqref{eq:20.17}.

\bigskip

\section{Proof of Theorem A}

We will give the proof of Theorem A  in the following two subsections.
\subsection {Proof of (i) of Theorem A}

 The main strategy of our proof is to decompose the operator into  a sum of the low frequency part and high frequency part by defining a measurable function $l_z:\ \mathbb{R}^2\rightarrow\mathbb{R}$ in \eqref{eq:2.6}. Here,   replacing $U_z$ by $2^{V_z}$ is a useful observation. For the low frequency part, we bound it by the sum of the Hardy-Littlewood maximal operator and  maximal truncated Hilbert transform. For the high frequency part, we further split it into two parts by the Littlewood-Paley projection in the first variable. The first part can also be controlled by the Hardy-Littlewood maximal operator. For the second part, we prove it by a local smoothing estimate which has been obtained in Section 2.

We first assume that $U(x_1,x_2)>0$ for almost every $(x_1,x_2)\in \mathbb{R}^{2}$ and the other case $U(x_1,x_2)<0$ can be handled similarly. Let $V:\ \mathbb{R}^2\rightarrow  \mathbb{Z}$ be a measurable function satisfying
\begin{align}\label{eq:2.3}
2^{V(x_1,x_2)}\leq U(x_1,x_2)< 2^{V(x_1,x_2)+1}.
\end{align}
For any given $k_0\in \mathbb{Z}$, we define
\begin{align*}
U^{(k_0)}(x_1,x_2):=2^{k_0}\frac{U(x_1,x_2)}{2^{V(x_1,x_2)}}.
\end{align*}
It is easy to see that $U^{(V(x_1,x_2))}(x_1,x_2)=U(x_1,x_2)$. Recall that $\psi:\ \mathbb{R}\rightarrow\mathbb{R}$ is a smooth function supported on $\{t\in \mathbb{R}:\ 1/2\leq |t|\leq 2\}$ with the property that $0\leq \psi(t)\leq 1$ and $\Sigma_{k\in \mathbb{Z}} \psi_k(t)=1$ for any $t\neq 0$, where $\psi_k(t)=\psi (2^{-k}t)$. For any given $l\in \mathbb{Z}$,  set
\begin{align}\label{eq:2.40}
H^l_{U,\gamma}f(x_1,x_2):=\mathrm{p.\,v.}\int_{-\infty}^{\infty}f(x_1-t,x_2-U(x_1,x_2)\gamma(t))\psi_l(t)\,\frac{\textrm{d}t}{t}.
\end{align}
Then, we can write
\begin{align*}
H_{U,\gamma}P^{(2)}_kf=\sum_{l\in \mathbb{Z}}H^l_{U,\gamma}P^{(2)}_kf.
\end{align*}

Furthermore, for simplicity, let $z:=(x_1,x_2)$, $U_z:=U(x_1,x_2)$ and $V_z:=V(x_1,x_2)$. Note that $\gamma$ is strictly increasing on $\mathbb{R}^+$ with $\gamma(\mathbb{R}^+)=\mathbb{R}^+$ from (ii) of {\bf(H.)} and \eqref{eq:1.5}. For these $k\in \mathbb{Z}$ and $V_z\in \mathbb{Z}$,  we can denote $l_z:\ \mathbb{R}^2\rightarrow\mathbb{R}$ be a measurable function satisfying
\begin{align}\label{eq:2.6}
2^{k}2^{V_z}\gamma(2^{l_z})=1.
\end{align}
We then further split $H_{U,\gamma}P^{(2)}_kf$ into the following low frequency part $H^{I}_{U,\gamma}P^{(2)}_kf$ and  the high frequency part $H^{II}_{U,\gamma}P^{(2)}_kf$, where
\begin{align*}
H^{I}_{U,\gamma}P^{(2)}_kf:=\sum_{l\leq l_z}H^l_{U,\gamma}P^{(2)}_kf \quad \textrm{and} \quad H^{II}_{U,\gamma}P^{(2)}_kf:=\sum_{l> l_z}H^l_{U,\gamma}P^{(2)}_kf.
\end{align*}

\textbf{Consider $H^{I}_{U,\gamma}P^{(2)}_kf$.} We compare it with the following operator $\mathbb{H}^{I}_{U,\gamma}P^{(2)}_kf$,
\begin{align*}
\mathbb{H}^{I}_{U,\gamma}P^{(2)}_kf(x_1,x_2)
:=\mathrm{p.\,v.}\int_{-\infty}^{\infty}P^{(2)}_kf(x_1-t,x_2)\phi_{l_z}(t)\,\frac{\textrm{d}t}{t},
\end{align*}
where  $\phi_{l_z}(t):=\sum_{l\leq l_z}\psi_l(t)$.
We can bound $|\mathbb{H}^{I}_{U,\gamma}P^{(2)}_kf|$ by
\begin{align*}
&\left|\mathrm{p.\,v.}\int_{-\infty}^{\infty}P^{(2)}_kf(x_1-t,x_2)(\phi_{l_z}(t)-1)\,\frac{\textrm{d}t}{t}\right|+\left|\mathrm{p.\,v.}\int_{-\infty}^{\infty}P^{(2)}_kf(x_1-t,x_2)\,\frac{\textrm{d}t}{t}\right|\\
\lesssim~ & M^{(1)}P^{(2)}_kf(x_1,x_2)+H^{*(1)}P^{(2)}_kf(x_1,x_2).
\end{align*}
Here and hereafter, $M^{(j)}$, $j=1,2$, denotes the \emph{Hardy-Littlewood maximal operator applied in the j-th variable}.  $H^{*(1)}$ denotes the \emph{maximal truncated Hilbert transform applied in the first variable}. Since both $M^{(1)}$ and $H^{*(1)}$ are known to be bounded on $L^p(\mathbb{R}^2)$, we may conclude that
\begin{align}\label{eq:2.7}
\left\|\mathbb{H}^{I}_{U,\gamma}P^{(2)}_kf\right\|_{L^{p}(\mathbb{R}^{2})}\ls \left\|P^{(2)}_kf\right\|_{L^{p}(\mathbb{R}^{2})}\quad \textrm{uniformly in}~k\in \mathbb{Z}~\textrm{for}~ \textrm{all}~p\in (1,\infty).
\end{align}

The difference between $H^{I}_{U,\gamma}P^{(2)}_kf(x_1,x_2)$ and $\mathbb{H}^{I}_{U,\gamma}P^{(2)}_kf(x_1,x_2)$ can be written as
\begin{align}\label{eq:2.8}
&\int_{-\infty}^{\infty}\left[\int_{0}^{U_z\gamma(t)}\partial_s\left(  P^{(2)}_kf(x_1-t,x_2-s)   \right)\,\textrm{d}s\right] \, \phi_{l_z}(t)\,\frac{\textrm{d}t}{t}\\
=&-\int_{-\infty}^{\infty}\left[\dashint_{0}^{U_z \gamma(t)} \bar{P}^{(2)}_kf(x_1-t,x_2-s) \,\textrm{d}s\right] 2^k U_z \gamma(t)   \phi_{l_z}(t)\,\frac{\textrm{d}t}{t},\nonumber
\end{align}
where $\bar{P}^{(2)}_kf$ denotes the \emph{Littlewood-Paley projection in the second variable corresponding to $((\cdot)\psi(\cdot))_k$}. From \eqref{eq:2.3} and \eqref{eq:2.6}, we have $2^k U_z\ls 1/\gamma(2^{l_z})$. Noticing \eqref{eq:1.5} implies $|\gamma(2^{l_z-j})/ \gamma(2^{l_z})|\ls e^{-C^{(1)}_{2}j/2}$, we can bound the last expression in \eqref{eq:2.8} by
\begin{align*}
\sum_{j\in \mathbb{N}}\dashint_{|t|\approx 2^{l_z-j}}\dashint_{0}^{U_z \gamma(t)} \left|\bar{P}^{(2)}_kf(x_1-t,x_2-s)\right| \,\textrm{d}s \left| \frac{\gamma(2^{l_z-j})}{\gamma(2^{l_z})}  \right| \,\textrm{d}t\ls M^{(1)}M^{(2)}\bar{P}^{(2)}_kf(x_1,x_2).
\end{align*}
From the $L^p(\mathbb{R})$-boundedness of $M^{(1)}$ and $M^{(2)}$, we have
\begin{align}\label{eq:2.9}
\left\|H^{I}_{U,\gamma}P^{(2)}_kf-\mathbb{H}^{I}_{U,\gamma}P^{(2)}_kf\right\|_{L^{p}(\mathbb{R}^{2})}\ls \left\|f\right\|_{L^{p}(\mathbb{R}^{2})}\quad \textrm{uniformly in}~k\in \mathbb{Z}~\textrm{for}~ \textrm{all}~p\in (1,\infty).
\end{align}
By replacing $f$ in \eqref{eq:2.9} with $P^{(2)}_kf$ and the fact that $P^{(2)}_kf=P^{(2)}_kP^{(2)}_kf$ essentially, we obtain the single annulus $L^p(\mathbb{R}^2)$-boundedness of $H^{I}_{U,\gamma}-\mathbb{H}^{I}_{U,\gamma}$. This, combined with \eqref{eq:2.7}, leads to
\begin{align}\label{eq:2.10}
\left\|H^{I}_{U,\gamma}P^{(2)}_kf\right\|_{L^{p}(\mathbb{R}^{2})}\ls \left\|P^{(2)}_kf\right\|_{L^{p}(\mathbb{R}^{2})}\quad \textrm{uniformly in}~k\in \mathbb{Z}~\textrm{for}~ \textrm{all}~p\in (1,\infty).
\end{align}

\textbf{Consider $H^{II}_{U,\gamma}P^{(2)}_kf$.} We rewrite it as $\sum_{l> 0}H^{l+l_z}_{U,\gamma}P^{(2)}_kf$. Hence, for $l\in \mathbb{N}$, it is enough to show that there exists a positive constant $\delta$ such that
\begin{align}\label{eq:2.11}
\left\|H^{l+l_z}_{U,\gamma}P^{(2)}_kf\right\|_{L^{p}(\mathbb{R}^{2})}\ls 2^{-\delta l}\|f\|_{L^{p}(\mathbb{R}^{2})}\quad  \textrm{for}~ \textrm{all}~l\in \mathbb{N}~\textrm{and}~p\in (2,\infty).
\end{align}
We split
\begin{align}\label{eq:2.12}
H^{l+l_z}_{U,\gamma}P^{(2)}_kf=\sum_{j\leq 0}H^{l+l_z}_{U,\gamma}P^{(1)}_{j-l_z-l}P^{(2)}_kf+\sum_{j\geq 1}H^{l+l_z}_{U,\gamma}P^{(1)}_{j-l_z-l}P^{(2)}_kf
=:H_1P^{(2)}_kf+H_2P^{(2)}_kf.
\end{align}

\medskip

\textbf{Consider $H_1P^{(2)}_kf$.} Let $Q^{(1)}_{-l_z-l}:=\sum_{j\leq 0}P^{(1)}_{j-l_z-l}$, then we bound
\begin{align}\label{n3.1}
|H_1P^{(2)}_kf|\leq \sup_{k_0\in \mathbb{Z}}\sup_{u\in[1,2)} \left| H^{l+l_0}_{2^{k_0}u,\gamma}Q^{(1)}_{-l_0-l}P^{(2)}_kf\right|\quad  \textrm{ with}~ 2^{k}2^{k_0}\gamma(2^{l_0})=1.
\end{align}
The multiplier of $H^{l+l_0}_{2^{k_0}u,\gamma}Q^{(1)}_{-l_0-l}P^{(2)}_kf$ is equal to
\begin{align*}
 \int_{-\infty}^{\infty}e^{-i\xi 2^{l+l_0}t-i\eta2^{k_0}u\gamma(2^{l+l_0}t) }\psi(t)\,\frac{\textrm{d}t}{t}\cdot\phi_{-l_0-l}(\xi)\cdot\psi_k(\eta),
\end{align*}
 where  $\phi_{-l_0-l}(\cdot):=\sum_{j\leq 0}\psi_{j-l_0-l}(\cdot)$.
Integration by parts about $e^{-i\eta2^{k_0}u\gamma(2^{l+l_0}t)}$ shows that we can further write it as the sum of the following two parts:
\begin{eqnarray*}
\left\{\aligned
\frac{1}{2^k2^{k_0}u \gamma(2^{l_0+l})}\int_{-\infty}^{\infty}e^{-i\xi 2^{l+l_0}t-i\eta2^{k_0}u\gamma(2^{l+l_0}t) }   \frac{\gamma(2^{l_0+l})\psi(t)}{ \gamma'(2^{l_0+l}t)2^{l_0+l}t}\,\textrm{d}t\cdot 2^{l+l_0}\xi \phi_{-l_0-l}(\xi)\cdot \frac{\psi_k(\eta)}{2^{-k}\eta};\\
\frac{i}{2^k2^{k_0}u \gamma(2^{l_0+l})}\int_{-\infty}^{\infty}e^{-i\xi 2^{l+l_0}t-i\eta2^{k_0}u\gamma(2^{l+l_0}t) }  \left(\frac{\psi(\cdot)\gamma(2^{l_0+l}) }{\gamma'(2^{l_0+l}\cdot)2^{l_0+l}\cdot  }\right)'(t)\,\textrm{d}t\cdot \phi_{-l_0-l}(\xi)\cdot \frac{\psi_k(\eta)}{2^{-k}\eta}.
\endaligned\right.
\end{eqnarray*}
We will only consider the first term above, since  the second term above can be handled similarly by noticing that $|t^2\gamma''(t)/ \gamma(t)|\leq C^{(2)}_3$ for any $t>0$.  With abusing notions, we write $H^{l+l_0}_{2^{k_0}u,\gamma}Q^{(1)}_{-l_0-l}P^{(2)}_kf(x_1,x_2)$ as
\begin{align*}
\frac{1}{2^k2^{k_0}u \gamma(2^{l_0+l})}\ \mathrm{p.\,v.}\int_{-\infty}^{\infty}\tilde{Q}^{(1)}_{-l_0-l}\tilde{P}^{(2)}_kf(x_1-t,x_2-2^{k_0}u\gamma(t))\frac{\gamma(2^{l_0+l})\psi_{l_0+l}(t)}{2^{l+l_0} \gamma'(t)}\,\frac{\textrm{d}t}{t},
\end{align*}
where\footnote{We denote $\mathcal{F}(f)$ means the \emph{Fourier transform} of $f$.} $\mathcal{F}(\tilde{Q}^{(1)}_{-l_0-l}f)(\xi,\eta):=\sum_{j\leq0}2^j\tilde{\psi}(2^{-j+l_0+l}\xi)\hat{f}(\xi,\eta)$ with $\tilde{\psi}(\cdot):=\cdot\psi(\cdot)$, and
$\mathcal{F}(\tilde{P}^{(2)}_kf)(\xi,\eta):= \frac{\psi_k(\eta)}{2^{-k}\eta}\hat{f}(\xi,\eta)$. Let $\mathbb{H}^{l+l_0}_{2^{k_0}u,\gamma}Q^{(1)}_{-l_0-l}P^{(2)}_kf(x_1,x_2)$ be
\begin{align}\label{eq:2.012}
\frac{1}{2^k2^{k_0}u \gamma(2^{l_0+l})}\ \mathrm{p.\,v.}\int_{-\infty}^{\infty}\tilde{Q}^{(1)}_{-l_0-l}\tilde{P}^{(2)}_kf(x_1,x_2-2^{k_0}u\gamma(t))\frac{\gamma(2^{l_0+l})\psi_{l_0+l}(t)}{2^{l+l_0} \gamma'(t)}\,\frac{\textrm{d}t}{t}.
\end{align}
 After changing of variable $2^{k_0}u\gamma(t)=:w$, we apply \eqref{eq:1.5}, (ii) of {\bf(H.)} and the fact $(\gamma^{-1})'(t)\gamma'(\gamma^{-1}(t))=1$ to obtain
\begin{align*}
 \frac{1}{\gamma'\left(\gamma^{-1}\left(w/2^{k_0}u \right)\right)}=
 \frac{\gamma\left(\gamma^{-1}\left(w/2^{k_0}u\right)\right)}{\gamma'\left(\gamma^{-1}\left(w/2^{k_0}u\right)\right)\gamma^{-1}\left(w/2^{k_0}u\right)}
 \frac{\gamma^{-1}\left(w/2^{k_0}u\right)}{w/2^{k_0}u}\ls\frac{2^{l_0+1}}{\gamma(2^{l_0+1})}.
\end{align*}
On the other hand,  note that $2^{k}2^{k_0}\gamma(2^{l_0})=1$ and \eqref{eq:1.5} imply $1/2^k2^{k_0}u \gamma(2^{l_0+l}) \ls e^{-C^{(1)}_{2}l/2}$. Then we have
$$
\left|\mathbb{H}^{l+l_0}_{2^{k_0}u,\gamma}Q^{(1)}_{-l_0-l}P^{(2)}_kf(x_1,x_2)\right|
\ls e^{-C^{(1)}_{2}l/2}M^{(2)}Q^{(1)}_{-l_0-l}P^{(2)}_kf(x_1,x_2)
\ls e^{-C^{(1)}_{2}l/2}M^{(2)}M^{(1)}M^{(2)}f(x_1,x_2).
$$
Therefore, the same boundedness also holds for
 $\sup_{k_0\in \mathbb{Z}}\sup_{u\in[1,2)} |\mathbb{H}^{l+l_0}_{2^{k_0}u,\gamma}Q^{(1)}_{-l_0-l}P^{(2)}_kf|,$
which trivially yields the estimate
\begin{align}\label{eq:2.13}
\left\|\sup_{k_0\in \mathbb{Z}}\sup_{u\in[1,2)}\left|\mathbb{H}^{l+l_0}_{2^{k_0}u,\gamma}Q^{(1)}_{-l_0-l}P^{(2)}_kf\right|\right\|_{L^{p}(\mathbb{R}^{2})}\ls e^{-C^{(1)}_{2}l/2}\|f\|_{L^{p}(\mathbb{R}^{2})}\quad  \textrm{for}~ \textrm{all}~p\in (1,\infty).
\end{align}

The difference between $H^{l+l_0}_{2^{k_0}u,\gamma}Q^{(1)}_{-l_0-l}P^{(2)}_kf(x_1,x_2)$ and $\mathbb{H}^{l+l_0}_{2^{k_0}u,\gamma}Q^{(1)}_{-l_0-l}P^{(2)}_kf(x_1,x_2)$ can be written as\footnote{We denote $\check{\tilde{\psi}}_{j-l_0-l}$ means the \emph{inverse Fourier transform} of $\tilde{\psi}_{j-l_0-l}$, i.e., $\check{\tilde{\psi}}_{j-l_0-l}(\cdot)=2^{j-l_0-l}\check{\tilde{\psi}}(2^{j-l_0-l}\cdot)$.}
\begin{align}\label{eq:2.14}
\sum_{j\leq0}\frac{2^j}{2^k2^{k_0}u \gamma(2^{l_0+l})}\int_{-\infty}^{\infty}
\Bigg[\int_{-\infty}^{\infty}&\tilde{P}^{(2)}_kf(x_1-w,x_2-2^{k_0}u\gamma(t))\\
\times &\left(\check{\tilde{\psi}}_{j-l_0-l}(w-t)-\check{\tilde{\psi}}_{j-l_0-l}(w)\right) \,\textrm{d}w\Bigg] \frac{\gamma(2^{l_0+l})\psi_{l_0+l}(t)}{2^{l+l_0} \gamma'(t)}\,\frac{\textrm{d}t}{t}.\nonumber
\end{align}
The mean value theorem gives
\begin{align}\label{eq:2.15}
\left| \check{\tilde{\psi}}_{j-l_0-l}(w-t)- \check{\tilde{\psi}}_{j-l_0-l}(w) \right|\lesssim |t|2^{2(j-l_0-l)}2^{-2m}
\end{align}
if $|t|\leq 2^{-j+l_0+l}$ and  $2^{-j+l_0+l+m-1}\leq |w|\leq 2^{-j+l_0+l+m}$ for $m\in \mathbb{N}$. For $m=0$ the above estimate holds for all $|w|\leq 2^{-j+l_0+l}$. Since $t\in \textrm{supp}~ \psi_{l_0+l} $ and $j\leq0$ imply  $|t|\leq 2^{-j+l_0+l}$,  thus  \eqref{eq:2.15} holds. Therefore the absolute value of \eqref{eq:2.14} can be controlled by
\begin{align*}
\sum_{m\in \mathbb{N}}\sum_{j\leq0}\frac{2^j}{2^k2^{k_0}u \gamma(2^{l_0+l})}\int_{-\infty}^{\infty}\Bigg[\int_{|w|\leq 2^{-j+l_0+l+m}}&\left|\tilde{P}^{(2)}_kf(x_1-w,x_2-2^{k_0}u\gamma(t))\right|\\
\times~ &|t|2^{2(j-l_0-l)}2^{-2m} \,\textrm{d}w\Bigg] \left|\frac{\gamma(2^{l_0+l})\psi_{l_0+l}(t)}{2^{l+l_0} \gamma'(t)}\right|\,\frac{\textrm{d}t}{|t|}.\nonumber
\end{align*}
By changing of variable and noticing that $2^{k}2^{k_0}\gamma(2^{l_0})=1$,  we apply \eqref{eq:1.5} and (ii) of {\bf(H.)} to bound the absolute value of \eqref{eq:2.14} by
\begin{align*}
\sum_{m\in \mathbb{N}}\sum_{j\leq0}\frac{2^{2j-m}}{2^k2^{k_0}u \gamma(2^{l_0+l})}M^{(1)}M^{(2)}\tilde{P}^{(2)}_kf(x_1,x_2)\ls e^{-C^{(1)}_{2}l/2}M^{(1)}M^{(2)}\tilde{P}^{(2)}_kf(x_1,x_2).
\end{align*}
Thus,  by the $L^p(\mathbb{R})$-boundedness of $M^{(1)}$, $M^{(2)}$ and $\tilde{P}^{(2)}_k$, we further obtain
\begin{align*}
\left\|\sup_{k_0\in \mathbb{Z}}\sup_{u\in[1,2)} \left|H^{l+l_0}_{2^{k_0}u,\gamma}Q^{(1)}_{-l_0-l}P^{(2)}_kf-\mathbb{H}^{l+l_0}_{2^{k_0}u,\gamma}
Q^{(1)}_{-l_0-l}P^{(2)}_kf\right|\right\|_{L^{p}(\mathbb{R}^{2})}\ls e^{-C^{(1)}_{2}l/2}\|f\|_{L^{p}(\mathbb{R}^{2})}\quad  \textrm{for}~ \textrm{all}~p\in (1,\infty).
\end{align*}
This, combined with \eqref{eq:2.13}, gives
\begin{align*}
\left\|\sup_{k_0\in \mathbb{Z}}\sup_{u\in[1,2)}\left|H^{l+l_0}_{2^{k_0}u,\gamma}Q^{(1)}_{-l_0-l}P^{(2)}_kf\right|\right\|_{L^{p}(\mathbb{R}^{2})}\ls e^{-C^{(1)}_{2}l/2}\|f\|_{L^{p}(\mathbb{R}^{2})}\quad  \textrm{for}~ \textrm{all}~p\in (1,\infty).
\end{align*}
By \eqref{n3.1}, we conclude  that
\begin{align}\label{eq:2.17}
\left\|H_1P^{(2)}_kf\right\|_{L^{p}(\mathbb{R}^{2})}\ls e^{-C^{(1)}_{2}l/2}\|f\|_{L^{p}(\mathbb{R}^{2})}\quad  \textrm{for}~ \textrm{all}~l\in \mathbb{N}~\textrm{and}~p\in (1,\infty).
\end{align}

\medskip

\textbf{Consider $H_2P^{(2)}_kf$.} It suffices to prove that there exists  $\delta>0$ such that
\begin{align}\label{eq:2.18}
\left\|H^{l+l_z}_{U,\gamma}P^{(1)}_{j-l_z-l}P^{(2)}_kf\right\|_{L^{p}(\mathbb{R}^{2})}\ls 2^{-\delta (j+l)}\|f\|_{L^{p}(\mathbb{R}^{2})}\quad  \textrm{for}~ \textrm{all}~p\in (2,\infty).
\end{align}
The expression inside the $L^{p}(\mathbb{R}^{2})$ norm on the LHS of \eqref{eq:2.18} can be majorized by
\begin{align*}
\sup_{k_0\in \mathbb{Z}}\left| H^{l+l_0}_{2^{k_0}U^{(0)},\gamma}P^{(1)}_{j-l_0-l}P^{(2)}_kf \right|\leq \Bigg(\sum_{k_0\in\mathbb{Z} }\left| H^{l+l_0}_{2^{k_0}U^{(0)},\gamma}P^{(1)}_{j-l_0-l}P^{(2)}_kf \right|^p\Bigg)^{\frac{1}{p}}\quad  \textrm{ with}~ 2^{k}2^{k_0}\gamma(2^{l_0})=1.
\end{align*}
Noting that the commutation relation of  $l^p$  and $L^p$ norms, and  the $l^p$ norm can be controlled by the $l^2$ norm for any $p\in(2,\infty)$, we will only  prove
\begin{align}\label{n3.2}
\left\| H^{l+l_0}_{2^{k_0}U^{(0)},\gamma}P^{(1)}_{j-l_0-l}P^{(2)}_kf \right\|_{L^{p}(\mathbb{R}^{2})}\ls 2^{-\delta (j+l)}\|f\|_{L^{p}(\mathbb{R}^{2})}\quad  \textrm{for}~ \textrm{all}~p\in (2,\infty).
\end{align}

To get \eqref{n3.2}, we will use the local smoothing estimate (Proposition \ref{proposition 2.1}).   Indeed,  by denoting $\Gamma_{l+l_0}(t):=\gamma(2^{l+l_0}t)/\gamma(2^{l+l_0})$ and $\phi(t):=\psi(t)/t$, we can rewrite $H^{l+l_0}_{2^{k_0}U^{(0)},\gamma}P^{(1)}_{j-l_0-l}P^{(2)}_kf(x_1,x_2)$  as
\begin{align*}
\int_{-\infty}^{\infty}P^{(1)}_{j-l_0-l}P^{(2)}_kf\left(x_1-2^{l+l_0}t,x_2-2^{k_0}\gamma(2^{l+l_0}) U^{(0)}(x_1,x_2)\Gamma_{l+l_0}(t)\right)\phi(t)\,\textrm{d}t.
\end{align*}
The reason of replacing $\gamma(2^{l+l_0}t)$ by $\Gamma_{l+l_0}(t)$ is that the main properties of   $\Gamma_{l+l_0}(t)$ is independent of $l$ and $l_0$ from Lemma \ref{lemma 2.1}. Let
\begin{align}\label{eq:2.20}
T^{l+l_0}f(x_1,x_2):=\int_{-\infty}^{\infty}f\left(x_1-t,x_2- U^{(0)}(x_1,x_2)\Gamma_{l+l_0}(t)\right)\phi(t)\,\textrm{d}t,
\end{align}
and for any given real numbers $a,b$, define
\begin{align*}
\Omega_{a,b}f(x_1,x_2):=f(2^ax_1,2^bx_2).
\end{align*}
Then
\begin{align*}
H^{l+l_0}_{2^{k_0}U^{(0)},\gamma}P^{(1)}_{j-l_0-l}P^{(2)}_kf=\Omega_{-l-l_0,-k_0-\log_2 \gamma(2^{l+l_0})}T^{l+l_0}\Omega_{l+l_0,k_0+\log_2 \gamma(2^{l+l_0})}P^{(1)}_{j-l_0-l}P^{(2)}_kf.
\end{align*}
Note that
\begin{align*}
\Omega_{l+l_0,k_0+\log_2 \gamma(2^{l+l_0})}P^{(1)}_{j-l_0-l}P^{(2)}_kf=P^{(1)}_{j}P^{(2)}_{k+k_0+\log_2 \gamma(2^{l+l_0})}\Omega_{l+l_0,k_0+\log_2 \gamma(2^{l+l_0})}f,
\end{align*}
so we can write
\begin{align*}
H^{l+l_0}_{2^{k_0}U^{(0)},\gamma}P^{(1)}_{j-l_0-l}P^{(2)}_kf=\Omega_{-l-l_0,-k_0-\log_2 \gamma(2^{l+l_0})}T^{l+l_0}P^{(1)}_{j}P^{(2)}_{k+k_0+\log_2 \gamma(2^{l+l_0})}\Omega_{l+l_0,k_0+\log_2 \gamma(2^{l+l_0})}f,
\end{align*}
which further gives
\begin{align}\label{eq:2.21}
\left\| H^{l+l_0}_{2^{k_0}U^{(0)},\gamma}P^{(1)}_{j-l_0-l}P^{(2)}_kf \right\|_{L^{p}(\mathbb{R}^{2})}=\left\| T^{l+l_0}P^{(1)}_{j}P^{(2)}_{k+k_0+\log_2 \gamma(2^{l+l_0})}f \right\|_{L^{p}(\mathbb{R}^{2})}.
\end{align}
Note that the frequency support of $P^{(1)}_{j}P^{(2)}_{k+k_0+\log_2 \gamma(2^{l+l_0})}f$ is contained in the annulus $\{(\xi,\eta)\in\mathbb{R}^{2}:\ \sqrt{\xi^2+\eta^2}\approx 2^{\max\{j,k+k_0+\log_2 \gamma(2^{l+l_0})\}}\}$.  We use  $2^{k}2^{k_0}\gamma(2^{l_0})=1$, (i) of {\bf(H.)} and  Proposition \ref{proposition 2.1} to obtain  that  there exists  $\delta>0$ such that
\begin{align*}
\left\|T^{l+l_0}P^{(1)}_{j}P^{(2)}_{k+k_0+\log_2 \gamma(2^{l+l_0})}f\right\|_{L^{p}(\mathbb{R}^{2})}\ls 2^{-\delta \max\{j,k+k_0+\log_2 \gamma(2^{l+l_0})\}}\|f\|_{L^{p}(\mathbb{R}^{2})}\ls 2^{-\frac{\delta j}{2}} e^{-C^{(1)}_{2}\delta l/4}\|f\|_{L^{p}(\mathbb{R}^{2})}
\end{align*}
for all $p\in (2,\infty)$. Therefore,
\begin{align*}
\left\|H_2P^{(2)}_kf\right\|_{L^{p}(\mathbb{R}^{2})}\ls e^{-C^{(1)}_{2}\delta l/4}\|f\|_{L^{p}(\mathbb{R}^{2})}\quad  \textrm{for}~ \textrm{all}~l\in \mathbb{N}~\textrm{and}~p\in (2,\infty).
\end{align*}
This, combined with \eqref{eq:2.17}, yields \eqref{eq:2.11}. Therefore, we finish the proof of (i) of  Theorem A.

\subsection{Proof of (ii) of Theorem A}

Since  the operator $M_{U,\gamma}$ that we are dealing is  positive, we may assume that $f$ is non-negative. Furthermore, we may assume that $U(x_1,x_2)>0$ for all $(x_1,x_2)\in \mathbb{R}^{2}$ and also adopt the notation
\begin{align*}
U^{(k_0)}(x_1,x_2)=2^{k_0}\frac{U(x_1,x_2)}{2^{V(x_1,x_2)}} \quad \textrm{for}~k_0\in \mathbb{Z},
\end{align*}
where $V:\ \mathbb{R}^2\rightarrow  \mathbb{Z}$ is a measurable function satisfying $2^{V(x_1,x_2)}\leq U(x_1,x_2)< 2^{V(x_1,x_2)+1}$. Recall that $\psi:\ \mathbb{R}\rightarrow\mathbb{R}$ is a smooth function supported on $\{t\in \mathbb{R}:\ 1/2\leq |t|\leq 2\}$ with the property that $0\leq \psi(t)\leq 1$ and $\Sigma_{k\in \mathbb{Z}} \psi_k(t)=1$ for any $t\neq 0$, where $\psi_k(t)=\psi (2^{-k}t)$.  Hence, $M_{U,\gamma}$ can be bounded by
\begin{align*}
   &\sup_{\varepsilon>0}\frac{1}{2\varepsilon} \sum_{l\in\mathbb{Z}:\ 2^l\leq \varepsilon}2^l \int_{-\infty}^{\infty}f(x_1-t,x_2-U(x_1,x_2)\gamma(t))\psi_l(t) \,\frac{\textrm{d}t}{|t|} \\
  \lesssim &  \sup_{l\in\mathbb{Z}} \int_{-\infty}^{\infty}f(x_1-t,x_2-U(x_1,x_2)\gamma(t))\psi_l(t) \,\frac{\textrm{d}t}{|t|}.
\end{align*}
For any $u>0$ and  $l\in\mathbb{Z}$, let
\begin{align}\label{eq:3.0}
   S_{u,l}f(x_1,x_2):= \int_{-\infty}^{\infty}f(x_1-t,x_2-u\gamma(t))\psi_l(t) \,\frac{\textrm{d}t}{|t|}.
\end{align}
As before, for simplify, we denote $z:=(x_1,x_2)$, $U_z:=U(x_1,x_2)$ and $V_z:=V(x_1,x_2)$. Then, by linearization, the $L^{p}(\mathbb{R}^{2})$-boundedness of $M_{U,\gamma}$ can be reduced to that of $S_{U_z,l_z}$, where $l_z:=l(x_1,x_2)$ and $l_z:\ \mathbb{R}^2\rightarrow\mathbb{Z}$ is a measurable function.

For these $V_z,l_z\in \mathbb{Z}$,  we denote $k_z:\ \mathbb{R}^2\rightarrow\mathbb{R}$ be a measurable function satisfying
\begin{align*}
2^{k_z}2^{V_z}\gamma(2^{l_z})=1,
\end{align*}
 then split $S_{U_z,l_z}f(z)$ as
\begin{align*}
S_{U_z,l_z}f(z)=\sum_{k\in\mathbb{Z} }S_{U_z,l_z}P^{(2)}_kf(z)=&\sum_{k\leq k_z}S_{U_z,l_z}P^{(2)}_kf(z)+\sum_{k> k_z}S_{U_z,l_z}P^{(2)}_kf(z)\\
=:&S^a_{U_z,l_z}f(z)+S^b_{U_z,l_z}f(z).
\end{align*}

\medskip

\textbf{Consider $S^a_{U_z,l_z}f(z)$.} Let us set
\begin{align*}
\mathbb{S}^a_{U_z,l_z}f(z):=\sum_{k\leq k_z}\int_{-\infty}^{\infty}P^{(2)}_kf(x_1-t,x_2)\psi_{l_z}(t) \,\frac{\textrm{d}t}{|t|},
\end{align*}
which can be majorized by $M^{(1)}(\sum_{k\leq k_z} P^{(2)}_kf)(z)$. Note that the operator $\sum_{k\leq k_z} P^{(2)}_kf$ is bounded on $L^p(\mathbb{R}^2)$ for any given $p\in (1,\infty)$ by multiplier theory, then we can use  the $L^p$ boundedness of $M^{(1)}$ to get
\begin{align}\label{eq:3.2}
\left\|\mathbb{S}^a_{U_z,l_z}f\right\|_{L^{p}(\mathbb{R}^{2})}\ls \|f\|_{L^{p}(\mathbb{R}^{2})}\quad  \textrm{for}~ \textrm{all}~p\in (1,\infty).
\end{align}
As in \eqref{eq:2.9}, we obtain the $L^p(\mathbb{R}^2)$-boundedness of $|S^a_{U_z,l_z}f-\mathbb{S}^a_{U_z,l_z}f|$. This, combined with \eqref{eq:3.2} shows that
\begin{align}\label{eq:3.4}
\left\|S^a_{U_z,l_z}f\right\|_{L^{p}(\mathbb{R}^{2})}\ls \|f\|_{L^{p}(\mathbb{R}^{2})}\quad  \textrm{for}~ \textrm{all}~p\in (1,\infty).
\end{align}

\medskip

\textbf{Consider $S^b_{U_z,l_z}f(z)$.} We rewrite it as $\sum_{k>0 }S_{U_z,l_z}P^{(2)}_{k+k_z}f(z)$. Furthermore, we split $S^b_{U_z,l_z}f(z)$ as the sum of $S^{1,b}_{U_z,l_z}f(z)$ and $S^{2,b}_{U_z,l_z}f(z)$, where
\begin{align*}
S^{1,b}_{U_z,l_z}f(z):= \sum_{m\leq -l_z }\sum_{k>0 }S_{U_z,l_z}P^{(1)}_{m}P^{(2)}_{k+k_z}f(z) \quad \textrm{and} \quad S^{2,b}_{U_z,l_z}f(z):= \sum_{m> -l_z }\sum_{k>0 }S_{U_z,l_z}P^{(1)}_{m}P^{(2)}_{k+k_z}f(z).
\end{align*}

\medskip

We first estimate $S^{1,b}_{U_z,l_z}f(z)$. Let
\begin{align*}
\mathbb{S}^{1,b}_{U_z,l_z}f(z):= \sum_{m\leq -l_z }\sum_{k>0 }\int_{-\infty}^{\infty}P^{(1)}_{m}P^{(2)}_{k+k_z}f(x_1,x_2-U_z\gamma(t))\psi_{l_z}(t) \,\frac{\textrm{d}t}{|t|}.
\end{align*}
After changing of variable as in \eqref{eq:2.012}, we  have

$$\left|\mathbb{S}^{1,b}_{U_z,l_z}f(z)\right|\ls M^{(2)}\Bigg(\sum_{m\leq -l_z }\sum_{k>0 }P^{(1)}_{m}P^{(2)}_{k+k_z}f\Bigg)(z),
$$
which gives
\begin{align}\label{eq:3.5}
\left\|\mathbb{S}^{1,b}_{U_z,l_z}f\right\|_{L^{p}(\mathbb{R}^{2})}\ls \|f\|_{L^{p}(\mathbb{R}^{2})}\quad  \textrm{for}~ \textrm{all}~p\in (1,\infty).
\end{align}
As in \eqref{eq:2.14}, the difference between $S^{1,b}_{U_z,l_z}f(z)$ and $\mathbb{S}^{1,b}_{U_z,l_z}f(z)$ can be written as
\begin{align*}
\sum_{m\leq -l_z }\sum_{k>0 }\int_{-\infty}^{\infty}\int_{-\infty}^{\infty}P^{(2)}_{k+k_z}f(x_1-w,x_2-U_z\gamma(t)) [\check{\psi}_m(w-t)-\check{\psi}_m(w)] \psi_{l_z}(t)\,\textrm{d}w\,\frac{\textrm{d}t}{|t|},
\end{align*}
 which can be controlled  by $M^{(1)}M^{(2)}(\sum_{m\leq -l_z }\sum_{k>0 } P^{(1)}_{m}P^{(2)}_{k+k_z}f)(z)$. Therefore, we obtain the $L^p(\mathbb{R}^2)$-boundedness of  $|S^{1,b}_{U_z,l_z}f-\mathbb{S}^{1,b}_{U_z,l_z}f|$ for $p>1$. This, combined with \eqref{eq:3.5}, yields
\begin{align}\label{eq:3.7}
\left\|S^{1,b}_{U_z,l_z}f\right\|_{L^{p}(\mathbb{R}^{2})}\ls \|f\|_{L^{p}(\mathbb{R}^{2})}\quad  \textrm{for}~ \textrm{all}~p\in (1,\infty).
\end{align}

 Next we turn to $S^{2,b}_{U_z,l_z}f(z)$.  We rewrite it as $\sum_{m> 0 }\sum_{k>0 }S_{U_z,l_z}P^{(1)}_{m-l_z}P^{(2)}_{k+k_z}f(z)$. Then, it suffices to prove that there exists a positive constant $\delta$ such that
\begin{align}\label{eq:3.8}
\left\|S_{U_z,l_z}P^{(1)}_{m-l_z}P^{(2)}_{k+k_z}f\right\|_{L^{p}(\mathbb{R}^{2})}\ls 2^{-\delta (k+m)}\|f\|_{L^{p}(\mathbb{R}^{2})}\quad  \textrm{for}~ \textrm{all}~p\in (2,\infty).
\end{align}
We further bound the LHS of \eqref{eq:3.8} by
\begin{align*}
\left\|\Bigg(\sum_{k_0\in\mathbb{Z}}\sum_{l\in\mathbb{Z} }\left| S_{U^{(k_0)}_z,l}P^{(1)}_{m-l}P^{(2)}_{k+K}f \right|^p\Bigg)^{\frac{1}{p}}\right\|_{L^{p}(\mathbb{R}^{2})},\quad  \textrm{where}~ 2^{K}2^{k_0}\gamma(2^l)=1.
\end{align*}
 The commutation relation of  $l^p$ and $L^p$ norms,  together with the fact that the $l^p$ norm can be controlled by the $l^2$ norm for all $p\in(2,\infty)$, implies that  \eqref{eq:3.8}  follows from
\begin{align}\label{eq:3.9}
\left\|S_{U^{(k_0)}_z,l}P^{(1)}_{m-l}P^{(2)}_{k+K}f\right\|_{L^{p}(\mathbb{R}^{2})}\ls 2^{-\delta (k+m)}\|f\|_{L^{p}(\mathbb{R}^{2})}\quad  \textrm{uniformly}~\textrm{in}~l\in\mathbb{Z}~\textrm{and}~\textrm{for} ~\textrm{all}~ p\in (2,\infty).
\end{align}
As in \eqref{eq:2.21}, the LHS of \eqref{eq:3.9} is equal to
\begin{align*}
\left\|T^l P^{(1)}_{m}P^{(2)}_{k+K+k_0+\log_2 \gamma(2^l)}f\right\|_{L^{p}(\mathbb{R}^{2})},
\end{align*}
where the definition of $T^l$ can be found in \eqref{eq:2.20}. Noting that $2^{K}2^{k_0}\gamma(2^l)=1$ and the frequency support of $P^{(1)}_{m}P^{(2)}_{k+K+k_0+\log_2 \gamma(2^l)}f$ is contained in the annulus $\{(\xi,\eta)\in\mathbb{R}^{2}:\ \sqrt{\xi^2+\eta^2}\approx\max\{2^{m},2^{k+K+k_0+\log_2 \gamma(2^l)}\}\approx 2^{\max\{m,k\}}\}$,  we can apply Proposition \ref{proposition 2.1} to get that there exists some $\delta_0>0$ such that
\begin{align*}
\left\|T^l P^{(1)}_{m}P^{(2)}_{k+K+k_0+\log_2 \gamma(2^l)}f\right\|_{L^{p}(\mathbb{R}^{2})}\ls 2^{-\delta_0 \max\{m,k\}}\|f\|_{L^{p}(\mathbb{R}^{2})}\ls 2^{-\delta_0 (k+m)/2}\|f\|_{L^{p}(\mathbb{R}^{2})},
\end{align*}
which implies \eqref{eq:3.9}. Therefore,
\begin{align}\label{eq:3.10}
\left\|S^{2,b}_{U_z,l_z}f\right\|_{L^{p}(\mathbb{R}^{2})}\ls \|f\|_{L^{p}(\mathbb{R}^{2})}\quad  \textrm{for}~ \textrm{all}~p\in (2,\infty).
\end{align}

Combining \eqref{eq:3.4}, \eqref{eq:3.7} and \eqref{eq:3.10}, we obtain
\begin{align*}
\left\|M_{U,\gamma}f\right\|_{L^{p}(\mathbb{R}^{2})}\ls \|f\|_{L^{p}(\mathbb{R}^{2})}\quad  \textrm{for}~ \textrm{all}~p\in (2,\infty),
\end{align*}
which finishes the proof of (ii) of Theorem A.

\section{Proof of Theorem B}

Before starting the proof of Theorem B, we first  study the $L^p(\mathbb{R}^2)$-boundedness, $1<p<\infty$, of the following \emph{maximal function associated with plane curve $(t,2^j\gamma(t))$ in lacunary coefficient},
\begin{align}\label{eq:4.0}
M_{L,\gamma}f(x_1,x_2):=\sup_{j\in \mathbb{Z}}\sup_{\varepsilon>0}\frac{1}{2\varepsilon}\int_{-\varepsilon}^{\varepsilon}|f(x_1-t,x_2-2^j\gamma(t))|\,\textrm{d}t
\end{align}
 under some weaker assumptions on $\gamma$ than {\bf (H.)},  which will play an important role in obtaining Theorem B. For the corresponding results about the operator $M_{L,\gamma}$ ,  we refer to \cite{HR, HKY}   for the case   $\gamma$ is a polynomial, and  \cite[Lemma 5.1]{GHLR}   for the case  $\gamma(t)=[t]^\alpha$ with $\alpha>0$ and $\alpha \neq 1$.   Here we  will provide a different and  slightly simpler proof.

\begin{proposition}\label{main result 5} Let $\gamma\in C(\mathbb{R})\bigcap C^{2}(\mathbb{R}^+)$ be either odd or even, $\gamma(0)=0$, and increasing on $\mathbb{R}^+$, and satisfying
\begin{enumerate}
  \item[\rm(i)] there exist  $C^{(1)}_2>0, C^{(1)}_3>0$ such that $C^{(1)}_2\leq|\frac{t\gamma'(t)}{\gamma(t)}|\leq C^{(1)}_3$ for any $t\in \mathbb{R}^+$;
  \item[\rm(ii)] there exists  $C^{(2)}_2>0$ such that $|\frac{t^2\gamma''(t)}{\gamma(t)}|\geq C^{(2)}_2$ for any $t\in \mathbb{R}^+$.
\end{enumerate}
Then, for any  $p>1$, there exists  $C>0$ such that
\begin{align*}
\|M_{L,\gamma}f\|_{L^{p}(\mathbb{R}^{2})}\leq C \|f\|_{L^{p}(\mathbb{R}^{2})}
\end{align*}
for any $f\in L^{p}(\mathbb{R}^{2})$.
\end{proposition}

\begin{proof}[Proof of Proposition \ref{main result 5}.] By linearization, it can be seen that $M_{L,\gamma}$ is the special situation of $M_{U,\gamma}$ with $U(x_1,x_2):=2^{j(x_1,x_2)}$, where $j(x_1,x_2):\ \mathbb{R}^2\rightarrow\mathbb{Z}$ is a measurable function. Therefore, we may use some ideas of the proof of (ii) of Theorem A. Recall that
\begin{align*}
   S_{u,l}f(x_1,x_2)= \int_{-\infty}^{\infty}f(x_1-t,x_2-u\gamma(t))\psi_l(t) \,\frac{\textrm{d}t}{|t|}
\end{align*}
in \eqref{eq:3.0}. Repeating the proof of (ii) of Theorem A, from \eqref{eq:3.4} and \eqref{eq:3.7}, as in \eqref{eq:3.8}, it suffices to prove that there exists a positive constant $\delta$ such that
\begin{align*}
\left\|S_{2^{j_z},l_z}P^{(1)}_{m-l_z}P^{(2)}_{k+k_z}f\right\|_{L^{p}(\mathbb{R}^{2})}\ls 2^{-\delta (k+m)}\|f\|_{L^{p}(\mathbb{R}^{2})}\quad  \textrm{for}~ \textrm{all}~p\in (1,\infty),
\end{align*}
where $k,m\in\mathbb{N} $ and $2^{k_z}2^{j_z}\gamma(2^{l_z})=1$. By interpolation, it is enough to prove
\begin{align}\label{eq:5.2}
\left\|S_{2^{j_z},l_z}P^{(1)}_{m-l_z}P^{(2)}_{k+k_z}f\right\|_{L^{p}(\mathbb{R}^{2})}\ls 2^{-\delta (k+m)}\|f\|_{L^{p}(\mathbb{R}^{2})}\quad  \textrm{for}~ \textrm{all}~p\in (2,\infty),
\end{align}
and
\begin{align}\label{eq:5.3}
\left\|S_{2^{j_z},l_z}P^{(1)}_{m-l_z}P^{(2)}_{k+k_z}f\right\|_{L^{p}(\mathbb{R}^{2})}\ls \|f\|_{L^{p}(\mathbb{R}^{2})}\quad  \textrm{for}~ \textrm{all}~p\in (1,2].
\end{align}

\textbf{Proof of \eqref{eq:5.2}.} Indeed, we have obtained \eqref{eq:5.2} in \eqref{eq:3.8} by a local smoothing estimate, where $\gamma$ satisfies {\bf(H.)}. For $M_{L,\gamma}$, we want to explore a simper method to get \eqref{eq:5.2}. We first bound the LHS of \eqref{eq:5.2} by
\begin{align*}
 \left\|\Bigg(\sum_{j_0\in\mathbb{Z} }\sum_{l\in\mathbb{Z} }\left| S_{2^{j_0},l}P^{(1)}_{m-l}P^{(2)}_{k+K}f \right|^p\Bigg)^{\frac{1}{p}}\right\|_{L^{p}(\mathbb{R}^{2})} \quad  \textrm{where}~2^{K}2^{j_0}\gamma(2^l)=1.
\end{align*}
As before,  it remains to prove
\begin{align}\label{eq:5.4}
\left\|S_{2^{j_0},l}P^{(1)}_{m-l}P^{(2)}_{k+K}f \right\|_{L^{p}(\mathbb{R}^{2})}\ls 2^{-\delta (k+m)}\|f\|_{L^{p}(\mathbb{R}^{2})}\quad  \textrm{uniformly}~\textrm{in}~l\in\mathbb{Z}~\textrm{for}~\textrm{all}~p\in (2,\infty).
\end{align}
As in \eqref{eq:2.21},  $2^{K}2^{j_0}\gamma(2^l)=1$  implies that the LHS of \eqref{eq:5.4} equals  $\|\mathbb{T}^l P^{(1)}_{m}P^{(2)}_{k}f\|_{L^{p}(\mathbb{R}^{2})}$, where
\begin{align*}
\mathbb{T}^lf(x_1,x_2):=\int_{-\infty}^{\infty}f(x_1-t,x_2- \Gamma_{l}(t))\psi(t)\,\frac{\textrm{d}t}{|t|}\quad \textrm{and} \quad \Gamma_{l}(t):=\frac{\gamma(2^lt)}{\gamma(2^l)}.
\end{align*}
Therefore, it suffices to show
\begin{align}\label{eq:5.5}
\left\|\mathbb{T}^l P^{(1)}_{m}P^{(2)}_{k}f\right\|_{L^{p}(\mathbb{R}^{2})}\ls 2^{-\delta (k+m)}\|f\|_{L^{p}(\mathbb{R}^{2})}\quad  \textrm{uniformly}~\textrm{in}~l\in\mathbb{Z}~\textrm{for}~\textrm{all}~p\in (2,\infty),
\end{align}
Repeating this process of proving \eqref{eq:20.15}, we  obtain \eqref{eq:5.5}.

\medskip

\textbf{Proof of \eqref{eq:5.3}.} By linearization, we bound the expression inside the $L^{p}(\mathbb{R}^{2})$ norm on the LHS of \eqref{eq:5.3} by $\sup_{l\in\mathbb{Z}}\sup_{j_0\in\mathbb{Z}}| S_{2^{j_0},l}P^{(1)}_{m-l}P^{(2)}_{k+K}f|$ with $2^{K}2^{j_0}\gamma(2^{l})=1$. Therefore, it suffices to show
\begin{align}\label{eq:n500}
 \left\| \Bigg(\sum_{l\in\mathbb{Z} }\left|\sup_{j_0\in\mathbb{Z}}\left| S_{2^{j_0},l}P^{(1)}_{m-l}P^{(2)}_{k+K}f\right|\right|^2\Bigg)^{\frac{1}{2}} \right\|_{L^{p}(\mathbb{R}^{2})}\ls \left\|\Bigg(\sum_{l\in\mathbb{Z} }\left| P^{(1)}_{m-l}f \right|^2\Bigg)^{\frac{1}{2}}\right\|_{L^{p}(\mathbb{R}^{2})}.
\end{align}
We will follow the method of bootstrapping an iterated interpolation argument in the spirit of Nagel, Stein and Wainger \cite{NSW}. Now consider the more general estimate
\begin{align}\label{eq:n50}
\left\| \Bigg(\sum_{l\in\mathbb{Z} }\left|\sup_{j_0\in\mathbb{Z}}\left| S_{2^{j_0},l}P^{(1)}_{m-l}P^{(2)}_{k+K}f\right|\right|^{q_1}\Bigg)^{\frac{1}{q_1}} \right\|_{L^{q_2}(\mathbb{R}^{2})}\ls \left\| \Bigg(\sum_{l\in\mathbb{Z}  }\left|P^{(1)}_{m-l}f\right|^{q_1}\Bigg)^{\frac{1}{q_1}} \right\|_{L^{q_2}(\mathbb{R}^{2})}
\end{align}
for some $q_1\in(1,\infty]$ and $q_2\in(1,\infty)$.

If $q_1=\infty$ and $q_2>2$, by linearization, we have that the expression inside the $L^{q_2}(\mathbb{R}^{2})$ norm on the LHS of \eqref{eq:n50} can be bounded by $M_{L,\gamma}M^{(2)}(\sup_{l\in\mathbb{Z}}|P^{(1)}_{m-l}f| )$. It is clear that \eqref{eq:5.2} implies the $L^{q_2}(\mathbb{R}^2)$-boundedness of $M_{L,\gamma}$ for all $q_2>2$. This along with the $L^{q_2}(\mathbb{R}^2)$-boundedness of $M^{(2)}$ yields
\begin{align}\label{n4.1}
\eqref{eq:n50}\  {\rm holds \   for} \  q_1=\infty \ {\rm and } \ q_2>2.
\end{align}
We claim
\begin{align}\label{eq:n51}
\left\|\sup_{j_0\in\mathbb{Z}}\left| S_{2^{j_0},l}P^{(1)}_{m-l}P^{(2)}_{k+K}f\right| \right\|_{L^{p}(\mathbb{R}^{2})}\ls\left \|f\right\|_{L^{p}(\mathbb{R}^{2})}\quad  \textrm{for}~ \textrm{all}~p\in (1,\infty).
\end{align}
If the claim above has been proved, then we can obtain \eqref{eq:n50} for the case $q_1=q_2\in(1,\infty)$ by replacing $f$ as $P^{(1)}_{m-l}f$, and using the commutation relation of $l^p$ and $L^p$ norms. Interpolation this with \eqref{n4.1} implies that \eqref{eq:n500} holds for all $p\in(4/3,2]$. Repeating the interpolation argument and using $q_1=\infty$ and $q_2\in(4/3,2]$, we can prove \eqref{eq:n500} holds for all $p\in(8/7,2]$. Reiterating this process sufficiently many times, we thus show \eqref{eq:n500} holds for all $p\in(1,2]$.

Therefore, it suffices to show $\eqref{eq:n51}$. Furthermore, note that $\sup_{j_0\in\mathbb{Z}}| S_{2^{j_0},l}P^{(1)}_{m-l}P^{(2)}_{k+K}f|$ can be bounded by $M_{L,\gamma}M^{(1)}M^{(2)}f$, and the $L^{p}(\mathbb{R}^2)$-boundedness of $M_{L,\gamma}$, $M^{(1)}$ and $M^{(2)}$ for all $p>2$, thus we may restrict $p\in(1,2]$ in the proof of $\eqref{eq:n51}$. This further reduces to showing
\begin{align}\label{eq:n600}
 \left\| \Bigg(\sum_{j_0\in\mathbb{Z} }\left| S_{2^{j_0},l}P^{(1)}_{m-l}P^{(2)}_{k+K}f\right|^2\Bigg)^{\frac{1}{2}} \right\|_{L^{p}(\mathbb{R}^{2})}\ls \left\|\Bigg(\sum_{j_0\in\mathbb{Z} }\left| P^{(2)}_{k+K}f \right|^2\Bigg)^{\frac{1}{2}}\right\|_{L^{p}(\mathbb{R}^{2})}\quad  \textrm{for}~ \textrm{all}~p\in (1,2].
\end{align}
Repeating the method of bootstrapping an iterated interpolation argument above, it suffices to prove
\begin{align}\label{eq:5.8}
\left\| S_{2^{j_0},l}P^{(1)}_{m-l}P^{(2)}_{k+K}f \right\|_{L^{p}(\mathbb{R}^{2})}\ls\left \| f\right\|_{L^{p}(\mathbb{R}^{2})}\quad  \textrm{for}~ \textrm{all}~p\in (1,\infty).
\end{align}
By Minkowski's inequality, it is easy to see that the LHS of \eqref{eq:5.8} can be bounded by
\begin{align*}
 \int_{-\infty}^{\infty} \left(\int_{-\infty}^{\infty}\int_{-\infty}^{\infty} \left|P^{(1)}_{m-l}P^{(2)}_{k+K}f(x_1-t,x_2- 2^{j_0}\gamma(t))\right |^p  \,\textrm{d}x_1 \,\textrm{d}x_2 \right)^{\frac{1}{p}} \psi_l(t)\,\frac{\textrm{d}t}{|t|}\ls \|f\|_{L^{p}(\mathbb{R}^{2})}.
\end{align*}
We then obtain \eqref{eq:5.8}. This finishes the proof of Propostion \ref{main result 5}.\end{proof}

\subsection {Proof of (i) of Theorem B}

Recall the definition of $H^l_{U,\gamma}$ in \eqref{eq:2.40}  and denote $\Delta:=\{l\in\mathbb{Z}:\ 2^{l}\leq \varepsilon_0 \}$, we then write
\begin{align*}
H^{\varepsilon_0}_{U,\gamma}P^{(2)}_kf=\sum_{l\in \Delta}H^l_{U,\gamma}P^{(2)}_kf.
\end{align*}
We further split $H^{\varepsilon_0}_{U,\gamma}P^{(2)}_kf$ as the sum of the following low frequency part $I^{\varepsilon_0}_{U,\gamma}P^{(2)}_kf$ and the high frequency part $II^{\varepsilon_0}_{U,\gamma}P^{(2)}_kf$,
\begin{align*}
I^{\varepsilon_0}_{U,\gamma}P^{(2)}_kf:=\sum_{l\in \Delta,\, l\leq l_z}H^l_{U,\gamma}P^{(2)}_kf \quad \textrm{and} \quad II^{\varepsilon_0}_{U,\gamma}P^{(2)}_kf:=\sum_{l\in \Delta,\, l> l_z}H^l_{U,\gamma}P^{(2)}_kf.
\end{align*}

We only consider the case  $l_z< \log_2 \varepsilon_0$,  since the other case $l_z\geq \log_2 \varepsilon_0$ can be handled similarly. The proof of $I^{\varepsilon_0}_{U,\gamma}P^{(2)}_kf$ can be found in \eqref{eq:2.10}. As for $II^{\varepsilon_0}_{U,\gamma}P^{(2)}_kf$, we may  rewrite it as $\sum_{0<l<\log_2\varepsilon_0- l_z}H^{l+l_z}_{U,\gamma}P^{(2)}_kf$. Therefore, for $l\in \mathbb{N}$ with $0<l<\log_2\varepsilon_0- l_z$, it  suffices to show that there exists a positive constant $\delta$ such that
\begin{align*}
\left\|H^{l+l_z}_{U,\gamma}P^{(2)}_kf\right\|_{L^{p}(\mathbb{R}^{2})}\ls 2^{-\delta l}\|f\|_{L^{p}(\mathbb{R}^{2})}\quad  \textrm{for}~ \textrm{all}~p\in (1,2].
\end{align*}
By interpolation, it is enough  to prove that there exists a positive constant $\delta$ such that
\begin{align}\label{eq:30.6}
\left\|H^{l+l_z}_{U,\gamma}P^{(2)}_kf\right\|_{L^{p}(\mathbb{R}^{2})}\ls 2^{-\delta l}\|f\|_{L^{p}(\mathbb{R}^{2})}\quad  \textrm{for}~ \textrm{all}~l\in \mathbb{N}~\textrm{and}~p\in (2,\infty),
\end{align}
and
\begin{align}\label{eq:30.7}
\left\|H^{l+l_z}_{U,\gamma}P^{(2)}_kf\right\|_{L^{p}(\mathbb{R}^{2})}\ls \|f\|_{L^{p}(\mathbb{R}^{2})}\quad  \textrm{for}~ \textrm{all}~l\in \mathbb{N}~\textrm{with}~0<l<\log_2\varepsilon_0- l_z~\textrm{and}~p\in (1,2].
\end{align}

The proof of \eqref{eq:30.6} has been obtained in \eqref{eq:2.11}. So it remains to show \eqref{eq:30.7}. From \eqref{eq:2.12} and \eqref{eq:2.17}, it suffices to prove that there exists a positive constant $\delta$ such that
\begin{align}\label{eq:30.14}
\left\|H^{l+l_z}_{U,\gamma}P^{(1)}_{j-l_z-l}P^{(2)}_kf\right\|_{L^{p}(\mathbb{R}^{2})}\ls 2^{-\delta j}\|f\|_{L^{p}(\mathbb{R}^{2})}\quad  \textrm{for}~ \textrm{all}~ p\in (1,2],
\end{align}
where $j,l\in \mathbb{N}$ with $0<l<\log_2\varepsilon_0- l_z$ and $j\geq 1$, and $2^{k}2^{V_z}\gamma(2^{l_z})=1$.

If we can prove
\begin{align}\label{eq:30.15}
\left\|H^{l+l_z}_{U,\gamma}P^{(1)}_{j-l_z-l}P^{(2)}_kf\right\|_{L^{p}(\mathbb{R}^{2})}\ls \|f\|_{L^{p}(\mathbb{R}^{2})}\quad  \textrm{for}~ \textrm{all}~ p\in (1,2],
\end{align}
then \eqref{eq:30.14} will follow by interpolation  between \eqref{eq:2.18} and \eqref{eq:30.15}.

Next, we will prove \eqref{eq:30.15}. Define a new measurable function $\tilde{U}^{(k_0)}_z:\ \mathbb{R}^2\rightarrow[2^{k_0}, 2^{k_0+1})$ as
\begin{align}\label{eq:30.16}
\tilde{U}^{(k_0)}_z:=U_z ~\textrm{if}~ V_z=k_0,  \quad \textrm{and} \quad\tilde{U}^{(k_0)}_z:=2^{k_0} ~\textrm{if} ~V_z\neq k_0.
\end{align}
We note that the property that both $\tilde{U}^{(k_0)}_z$ and $U^{(k_0)}_z$ are  in $[2^{k_0}, 2^{k_0+1})$  is very important to us. Then the expression inside the $L^{p}(\mathbb{R}^{2})$ norm on the LHS of \eqref{eq:30.15} can be bounded by
\begin{align*}
\sup_{k_0 \in \mathbb{Z}}\left|H^{l+l_0}_{\tilde{U}^{(k_0)}_z,\gamma}P^{(1)}_{j-l_0-l}P^{(2)}_kf\right|, \quad \textrm{where}~ 2^{k}2^{k_0}\gamma(2^{l_0})=1.
\end{align*}
Furthermore, let $\Lambda:=\{l\in\mathbb{Z}:\ 0<l<\log_2\varepsilon_0- l_0 \}$, we then bound the LHS of \eqref{eq:30.15} by
\begin{align*}
\left\| \Bigg(\sum_{k_0\in\mathbb{Z} }\left|\mathbf{1}_\Lambda (l) H^{l+l_0}_{\tilde{U}^{(k_0)}_z,\gamma}P^{(1)}_{j-l_0-l}P^{(2)}_kf\right|^2\Bigg)^{\frac{1}{2}} \right\|_{L^{p}(\mathbb{R}^{2})}, \quad \textrm{where}~2^{k}2^{k_0}\gamma(2^{l_0})=1.
\end{align*}
Therefore, for $p\in (1,2]$, it is enough to prove that
\begin{align}\label{eq:30.17}
\left\| \Bigg(\sum_{k_0\in\mathbb{Z} }\left|\mathbf{1}_\Lambda (l) H^{l+l_0}_{\tilde{U}^{(k_0)}_z,\gamma}P^{(1)}_{j-l_0-l}P^{(2)}_kf\right|^2\Bigg)^{\frac{1}{2}} \right\|_{L^{p}(\mathbb{R}^{2})}\ls \left\| \Bigg(\sum_{k_0\in\mathbb{Z}  }\left|P^{(1)}_{j-l_0-l}f\right|^2\Bigg)^{\frac{1}{2}} \right\|_{L^{p}(\mathbb{R}^{2})}.
\end{align}
 Note that $|\mathbf{1}_\Lambda (l) H^{l+l_0}_{\tilde{U}^{(k_0)}_z,\gamma}P^{(1)}_{j-l_0-l}P^{(2)}_kf|$ can be bounded by $M_{U,\gamma}M^{(2)}(\sup_{k_0\in\mathbb{Z}}|P^{(1)}_{j-l_0-l}f| )$ and apply (ii) of Theorem A,  then we can repeat  this process from $\eqref{eq:n500}$ to $\eqref{eq:n51}$.  Thus estimate \eqref{eq:30.17} reduces to proving
\begin{align}\label{eq:30.19}
\left\|\mathbf{1}_\Lambda (l) H^{l+l_0}_{\tilde{U}^{(k_0)}_z,\gamma}P^{(1)}_{j-l_0-l}P^{(2)}_kf \right\|_{L^{p}(\mathbb{R}^{2})}\ls\left \| f\right\|_{L^{p}(\mathbb{R}^{2})}\quad  \textrm{for}~ \textrm{all}~p\in (1,\infty).
\end{align}

In fact,  we  will prove  the following   stronger version,
\begin{align*}
\left\|\mathbf{1}_\Lambda (l) H^{l+l_0}_{\tilde{U}^{(k_0)}_z,\gamma}f \right\|_{L^{p}(\mathbb{R}^{2})}\ls\left \| f\right\|_{L^{p}(\mathbb{R}^{2})}\quad  \textrm{for}~ \textrm{all}~p\in (1,\infty).
\end{align*}
We now define another new measurable function  $\tilde{\tilde{U}}^{(k_0)}_z:\ \mathbb{R}^2\rightarrow[2^{k_0}, 2^{k_0+1})$ as follows:
\begin{align}\label{eq:30.21}
\tilde{\tilde{U}}^{(k_0)}_z:=U_z~ \textrm{if}~ V_z=k_0, ~\textrm{and}~ \textrm{extend} ~\tilde{\tilde{U}}^{(k_0)}_z ~\textrm{to the whole} ~\mathbb{R}^{2}~ \textrm{space with} ~ \|\tilde{\tilde{U}}^{(k_0)}\|_{\textrm{Lip}}\leq 2 \|U\|_{\textrm{Lip}}.
\end{align}
 We remark that both $\tilde{\tilde{U}}^{(k_0)}_z$ and $\tilde{U}^{(k_0)}_z$ and are in $[2^{k_0}, 2^{k_0+1})$, but  the former is  Lipschitz and the latter is not. By the definitions of $M_{L,\gamma}$, $\tilde{U}^{(k_0)}_z$  and $\tilde{\tilde{U}}^{(k_0)}_z$,  one has the following pointwise estimate
\begin{align}\label{eq:30.y}
\mathbf{1}_\Lambda (l) H^{l+l_0}_{\tilde{U}^{(k_0)}_z,\gamma}f(z)\ls M_{L,\gamma}f(z)+\mathbf{1}_\Lambda (l) H^{l+l_0}_{\tilde{\tilde{U}}^{(k_0)}_z,\gamma}f(z).
\end{align}

By  Propostion \ref{main result 5}, it suffices to prove
\begin{align}\label{eq:30.23}
\left\| \mathbf{1}_\Lambda (l) H^{l+l_0}_{\tilde{\tilde{U}}^{(k_0)}_z,\gamma}f\right\|_{L^{p}(\mathbb{R}^{2})}\ls\left \| f\right\|_{L^{p}(\mathbb{R}^{2})}\quad  \textrm{for}~ \textrm{all}~p\in (1,\infty).
\end{align}
Indeed, by Minkowski's inequality, the LHS of \eqref{eq:30.23} can be bounded by
\begin{align}\label{eq:30.24}
\mathbf{1}_\Lambda (l) \int_{-\infty}^{\infty} \left(\int_{-\infty}^{\infty}\int_{-\infty}^{\infty} \left|f(x_1-t,x_2-\tilde{\tilde{U}}^{(k_0)}_z\gamma(t))\right |^p  \,\textrm{d}x_1 \,\textrm{d}x_2 \right)^{\frac{1}{p}}     \psi_{l+l_0}(t)\,\frac{\textrm{d}t}{|t|}.
\end{align}
Let
\begin{align}\label{eq:30.25}
X_1:=x_1-t  \quad   \textrm{and} \quad   X_2:=x_2-\tilde{\tilde{U}}^{(k_0)}_z\gamma(t),
\end{align}
we then write the corresponding Jacobian determinant as
\begin{align*}
\frac{\partial(X_1,X_2)}{\partial(x_1,x_2)}=
\left|
\begin{array}{ccc}
 1~~&~~  0  \\
 - \frac{\partial}{\partial_{x_1}}\tilde{\tilde{U}}^{(k_0)}_z \gamma(t)   ~~&~~1- \frac{\partial}{\partial_{x_2}}\tilde{\tilde{U}}^{(k_0)}_z \gamma(t)
\end{array}
\right|=1- \frac{\partial}{\partial_{x_2}}\tilde{\tilde{U}}^{(k_0)}_z \gamma(t).
\end{align*}
Noting that $t\in \textrm{supp}~\psi_{l+l_0}$, $l\in \Delta$ and $|\frac{\partial}{\partial_{x_2}}\tilde{\tilde{U}}^{(k_0)}_z|\leq \|\tilde{\tilde{U}}^{(k_0)}\|_{\textrm{Lip}}\leq 2 \|U\|_{\textrm{Lip}}$, we assert that
\begin{align*}
\left|\frac{\partial}{\partial_{x_2}}\tilde{\tilde{U}}^{(k_0)}_z \gamma(t)\right|\leq 2 \|U\|_{\textrm{Lip}}\gamma(2 \varepsilon_0).
\end{align*}
Furthermore, let the positive constant $\varepsilon_0$ satisfying $\gamma(2 \varepsilon_0)\leq 1/4\|U\|_{\textrm{Lip}}$,  then we  conclude that
\begin{align*}
\left|\frac{\partial(X_1,X_2)}{\partial(x_1,x_2)}\right|\geq \frac{1}{2}.
\end{align*}
This implies that the change of variables in \eqref{eq:30.25} is valid and thus \eqref{eq:30.24} can be bounded by $\| f\|_{L^{p}(\mathbb{R}^{2})}$ for all $p\in(1,\infty)$. Therefore, we obtain \eqref{eq:30.23}, which completes the proof of (i) of Theorem B.

\subsection{Proof of (ii) of Theorem B}

From \eqref{eq:3.4} and \eqref{eq:3.7}, we have obtained the $L^p(\mathbb{R}^2)$-boundedness of $S^a_{U_z,l_z}f(z)$ and $S^{1,b}_{U_z,l_z}f(z)$ for all $p\in(1,\infty)$. But it is possible that the operator $S^{2,b}_{U_z,l_z}f(z)$ is  unbounded on $L^p(\mathbb{R}^2)$ for any $p\in(1,2]$, if we only assume that $U$ is a measurable function. Therefore, for the case $p\in(1,2]$, instead of $M_{U,\gamma}$, we consider $M^{\varepsilon_0}_{U,\gamma}$ with $U$ is a Lipschitz function. The truncation $\varepsilon_0$ depending only on $\|U\|_{\textrm{Lip}}$ plays a crucial role in the proof of the $L^p(\mathbb{R}^2)$-boundedness of $S^{2,b}_{U_z,l_z}f(z)$ for all $p\in(1,2]$, which implies that the measurable function $l_z:\ \mathbb{R}^2\rightarrow\mathbb{Z}$ satisfies $2^{l_z}\leq \varepsilon_0$.

We now turn to $S^{2,b}_{U_z,l_z}$, i.e.,
\begin{align*}
S^{2,b}_{U_z,l_z}f(z)=\sum_{m> 0 }\sum_{k>0 }S_{U_z,l_z}P^{(1)}_{m-l_z}P^{(2)}_{k+k_z}f(z).
\end{align*}
Indeed, by interpolation with \eqref{eq:3.8}, it suffices to show
\begin{align}\label{eq:4.1}
\left\|S_{U_z,l_z}P^{(1)}_{m-l_z}P^{(2)}_{k+k_z}f\right\|_{L^{p}(\mathbb{R}^{2})}\ls \|f\|_{L^{p}(\mathbb{R}^{2})}\quad  \textrm{for}~ \textrm{all}~p\in (1,2].
\end{align}

Recall the measurable function $\tilde{U}^{(k_0)}_z:\ \mathbb{R}^2\rightarrow[2^{k_0}, 2^{k_0+1})$ defined in \eqref{eq:30.16}. Then the expression inside the $L^{p}(\mathbb{R}^{2})$ norm on the LHS of \eqref{eq:4.1} can be bounded by
\begin{align*}
\sup_{k_0 \in \mathbb{Z}}\left|S_{\tilde{U}^{(k_0)}_z,l_z}P^{(1)}_{m-l_z}P^{(2)}_{k+k_z}f\right|, \quad \textrm{where}~ 2^{k_z}2^{k_0}\gamma(2^{l_z})=1.
\end{align*}
Noting $2^{l_z}\leq \varepsilon_0$ and letting $\Delta:=\{l\in\mathbb{Z}:\ 2^{l}\leq \varepsilon_0 \}$, we bound the LHS of \eqref{eq:4.1} by
\begin{align*}
\left\| \Bigg(\sum_{l\in\mathbb{Z} }\left|\mathbf{1}_\Delta (l)\sup_{k_0\in\mathbb{Z}}\left| S_{\tilde{U}^{(k_0)}_z,l}P^{(1)}_{m-l}P^{(2)}_{k+K}f\right|\right|^2\Bigg)^{\frac{1}{2}} \right\|_{L^{p}(\mathbb{R}^{2})}, \quad \textrm{where}~2^{K}2^{k_0}\gamma(2^{l})=1.
\end{align*}
Therefore, for $p\in (1,2]$, it is enough to prove that
\begin{align}\label{eq:4.3}
\left\|\Bigg(\sum_{l\in\mathbb{Z} }\left|\mathbf{1}_\Delta (l)\sup_{k_0\in\mathbb{Z}}\left| S_{\tilde{U}^{(k_0)}_z,l}P^{(1)}_{m-l}P^{(2)}_{k+K}f\right|\right|^2\Bigg)^{\frac{1}{2}} \right\|_{L^{p}(\mathbb{R}^{2})}\ls \left\| \Bigg(\sum_{l\in\mathbb{Z}  }\left|P^{(1)}_{m-l}f\right|^2\Bigg)^{\frac{1}{2}} \right\|_{L^{p}(\mathbb{R}^{2})}.
\end{align}
Notice that $\sup_{l\in\mathbb{Z}}\mathbf{1}_\Delta (l)\sup_{k_0\in\mathbb{Z}}| S_{\tilde{U}^{(k_0)}_z,l}P^{(1)}_{m-l}P^{(2)}_{k+K}f|$ can be bounded by $M_{U,\gamma}M^{(2)}(\sup_{l\in\mathbb{Z}}|P^{(1)}_{m-l}f| )$, and the $L^p(\mathbb{R}^2)$-boundedness, $p\in (2,\infty)$, of $M_{U,\gamma}$ has been obtained in (ii) of Theorem A. Repeating this process from $\eqref{eq:n500}$ to $\eqref{eq:n51}$,  it is enough to show
\begin{align}\label{eq:4.5}
\left\| \mathbf{1}_\Delta (l)\sup_{k_0\in\mathbb{Z}}\left| S_{\tilde{U}^{(k_0)}_z,l}P^{(1)}_{m-l}P^{(2)}_{k+K}f\right| \right\|_{L^{p}(\mathbb{R}^{2})}\ls\left \| f\right\|_{L^{p}(\mathbb{R}^{2})}\quad  \textrm{for}~ \textrm{all}~p\in (1,\infty).
\end{align}
Furthermore, it suffices to prove \eqref{eq:4.5} for $p\in(1,2]$. This is because $\sup_{k_0\in\mathbb{Z}} S_{\tilde{U}^{(k_0)}_z,l}$ can be bounded by $M_{U,\gamma}$ and the result in (ii) of Theorem A. As in \eqref{eq:4.3}, we  only  need to show that
\begin{align*}
\left\| \Bigg(\sum_{k_0\in\mathbb{Z} }\left|\mathbf{1}_\Delta (l)S_{\tilde{U}^{(k_0)}_z,l}P^{(1)}_{m-l}P^{(2)}_{k+K}f\right|^2\Bigg)^{\frac{1}{2}} \right\|_{L^{p}(\mathbb{R}^{2})}\ls \left\| \Bigg(\sum_{k_0\in\mathbb{Z}  }\left|P^{(2)}_{k+K}f\right|^2\Bigg)^{\frac{1}{2}} \right\|_{L^{p}(\mathbb{R}^{2})}\quad \textrm{where}~2^{K}2^{k_0}\gamma(2^{l})=1.
\end{align*}
By the same argument as in \eqref{eq:4.3}-\eqref{eq:4.5}, it remains to show that
\begin{align*}
\left\| \mathbf{1}_\Delta (l)S_{\tilde{U}^{(k_0)}_z,l}P^{(1)}_{m-l}P^{(2)}_{k+K}f\right\|_{L^{p}(\mathbb{R}^{2})}\ls\left \| f\right\|_{L^{p}(\mathbb{R}^{2})}\quad  \textrm{for}~ \textrm{all}~p\in (1,\infty).
\end{align*}

In fact we will prove the following stronger version,
\begin{align*}
\left\| \mathbf{1}_\Delta (l)S_{\tilde{U}^{(k_0)}_z,l}f\right\|_{L^{p}(\mathbb{R}^{2})}\ls\left \| f\right\|_{L^{p}(\mathbb{R}^{2})}\quad  \textrm{for}~ \textrm{all}~p\in (1,\infty).
\end{align*}
As in \eqref{eq:30.y}, we have the following pointwise estimate
\begin{align*}
\mathbf{1}_\Delta (l)S_{\tilde{U}^{(k_0)}_z,l}f(z)\leq M_{L,\gamma}f(z)+\mathbf{1}_\Delta (l)S_{\tilde{\tilde{U}}^{(k_0)}_z,l}f(z).
\end{align*}
From Propostion \ref{main result 5}, we only have to show that
\begin{align}\label{eq:4.9}
\left\| \mathbf{1}_\Delta (l)S_{\tilde{\tilde{U}}^{(k_0)}_z,l}f\right\|_{L^{p}(\mathbb{R}^{2})}\ls\left \| f\right\|_{L^{p}(\mathbb{R}^{2})}\quad  \textrm{for}~ \textrm{all}~p\in (1,\infty).
\end{align}
As in \eqref{eq:30.23}, noting that $t\in \textrm{supp}~\psi_l$, $l\in \Delta$ and $|\frac{\partial}{\partial_{x_2}}\tilde{\tilde{U}}^{(k_0)}_z|\leq \|\tilde{\tilde{U}}^{(k_0)}\|_{\textrm{Lip}}\leq 2 \|U\|_{\textrm{Lip}}$ imply $|\frac{\partial}{\partial_{x_2}}\tilde{\tilde{U}}^{(k_0)}_z \gamma(t)|\leq 2 \|U\|_{\textrm{Lip}}\gamma(2 \varepsilon_0)$, and the fact that $\varepsilon_0$ satisfies $\gamma(2 \varepsilon_0)\leq 1/4\|U\|_{\textrm{Lip}}$, we then obtain \eqref{eq:4.9}. This completes the proof of (ii) of Theorem B.

\bigskip

\noindent{\bf Acknowledgments.} We would like to thank Prof. Lixin Yan and Prof. Junfeng Li for helpful suggestions and discussions. We also particularly thank Dr.  Shaoming Guo for valuable discussions and providing us some important references which is very useful to our work.

\bigskip

\noindent Naijia Liu, Liang Song and Haixia Yu (Corresponding author)

\smallskip

\noindent School of Mathematics, Sun Yat-sen University, Guangzhou, 510275,  People's Republic of China

\smallskip

\noindent{\it E-mail}: \texttt{liunj@mail2.sysu.edu.cn} (N. Liu)

\noindent{\it E-mail}: \texttt{songl@mail.sysu.edu.cn} (L. Song)

\noindent{\it E-mail}: \texttt{yuhx26@mail.sysu.edu.cn} (H. Yu)

\bigskip

\end{document}